\documentclass[11pt,reqno]{amsart}

\usepackage{amsmath,amssymb,mathrsfs}
\usepackage{graphicx,cite,times}
\usepackage{cases}
\usepackage{color}
\usepackage{appendix}

\newtheorem{theo}{Theorem}[section]

\newtheorem{lemm}[theo]{Lemma}

\newtheorem{rema}[theo]{Remark}

\newtheorem{assu}{Hypothesis}
\numberwithin{equation}{section}

\begin{document}

\title [Response solutions for  wave equations]{Response solutions for wave equations with variable
wave speed and periodic forcing}

\author{Bochao Chen}
\address{School of Mathematics, Jilin University, Changchun, Jilin 130012, P.R.China}
\email{chenbc758@163.com}

\author{Yixian Gao}
\address{School of Mathematics and Statistics, Center for Mathematics and
Interdisciplinary Sciences, Northeast Normal University, Changchun, Jilin
130024, P.R.China}
\email{gaoyx643@nenu.edu.cn}

\author{Yong  Li}
\address{School of Mathematics, Jilin University, Changchun, Jilin 130012;
School of Mathematics and Statistics, Center for Mathematics and
Interdisciplinary Sciences, Northeast Normal University, Changchun, Jilin
130024, P.R.China}
\email{yongli@nenu.edu.cn}

\author{Xue Yang}
\address{School of Mathematics, Jilin University, Changchun, Jilin 130012;
School of Mathematics and Statistics, Center for Mathematics and
Interdisciplinary Sciences, Northeast Normal University, Changchun, Jilin
130024, P.R.China}
\email{yangx100@nenu.edu.cn}

\thanks{The research of  BC was supported in part by NSFC grant 11901232 and the China Postdoctoral Science Foundation grant  2019M651191;
The research of YG was supported in part by NSFC grant 11871140,  11671071, JJKH 20180006KJ and FRFCU 2412019BJ005;
The research of YL was supported in part by NSFC grant 11571065.}



\begin{abstract}

We consider a model of   nonlinear wave equations with periodically varying wave speed and periodic external forcing. By imposing non-resonance
conditions on the frequency,  we establish the existence of  the response solutions (i.e., periodic solutions with the same frequency as the forcing) for such a model in a Cantor set of asymptotically full measure. The proof relies on a  Lyapunov--Schmidt reduction together with the Nash--Moser iteration.
\end{abstract}

\maketitle

\section{Introduction}

The traditional underlying model for wave propagation, in general, might be described by  wave equations
with coefficients which are independent of time.
 However, more and more models with time-dependent coefficients have been found in the scientific community,
 such as the fading and modulation problems in telecommunications,  microwave propagation in ferrites, propagation of electromagnetic waves in time memory medium and so on.
In the present paper we are concerned with  nonlinear wave equations with  periodic time-dependent coefficient
\begin{align}\label{E1.1}
u_{tt}-a(\omega t) u_{xx}=\epsilon f(x,u)+\epsilon g(\omega t,x),\quad t\in\mathbb{R}
\end{align}
satisfying the  periodic boundary conditions. The external force term $g$ is periodically varying  with the frequency parameter $\omega \in \mathbb R^{+}$.
The time-dependent coefficient $a(\omega t)>0$ denotes the wave (or  propagation) speed. The positive parameter $\epsilon$ is  small enough. And the nonlinearity $f(x,u)$ is equal to $0$ at $u = 0$. When the wave speed $a$ is periodically varying, our purpose is to give  the existence of  periodic solutions whose frequency is consistent with the one of the periodic external forcing.

For equation \eqref{E1.1} with $\omega=1,f=0,g=0$, the linear model with time-dependent wave speed
\begin{align}\label{lp}
\textstyle u_{tt} -a(t)u_{xx}=0
\end{align}
 has been investigated as a linearized model of the Kirchhoff equation, which describes
the vibration of an elastic string \cite{kh1994}. The effect of the
time-dependent coefficient is crucial for the asymptotic behavior of the solution (\cite{Hirosawa2007Wave}). One of the main tasks for such problems is to give the precise analysis of solutions by taking into account the properties of the coefficient $a$. There have been many results from the point of view of the Cauchy problem. A pioneer consideration of local well-posedness results associated with \eqref{lp} subject to non-Lipschitz continuous coefficients  was given by Colombini et al. \cite{Colombini1979Sur}.  They described the Gevrey and the $C^\infty$ well-posedness of  \eqref{lp} with respect to H\"{o}lder and Log-Lipschitz continuity of the coefficients. In \cite{Hirosawa2007Wave,Hirosawa2010wave}, Hirosawa established the energy estimates for $d$-dimensional wave equations
\begin{align}\label{particular}
\textstyle u_{tt}-a(t)\Delta u=0,
\end{align}
corresponding to  the $C^k$ and Gevrey regularity of $a(t)$. Generally, the theory of the classical wave equation can be generalized to the one of second order linear strictly hyperbolic equations with time-dependent  coefficients, that is,
\begin{align*}
\textstyle u_{tt}-\sum^{d}_{i,j=1}a_{ij}(t) u_{x_{i}x_{j}}=0,
\end{align*}
where $\sum^{d}_{i,j=1}a_{ij}(t)\zeta_i\zeta_j\geq a_0|\zeta|^2,\forall\zeta\in\mathbb{R}^d$, with $a_0>0$. In particular, for $a_{ii}(t)=a (t)$ and $a_{ij}(t)=0, i\neq j$,  we simplify it to equation \eqref{particular}. Some recent results
on hyperbolic systems with time-dependent wave speed can be found in the literature  \cite{Reissig2011hyperbolic,Colombini2015hyperbolic}. Moreover, it is worth mentioning  that Yagdjian \cite{yagdjian2005global} considered dissipative wave equations with time-dependent coefficients
\begin{align}\label{periodic}
\textstyle u_{tt}-a(t)\Delta u+f(u)(u^2_t-a(t)\sum^n_{k=1}(u_{x_k})^2)=0, \quad x\in \mathbb {R}^d.
\end{align}
If $a(t)$ is a periodic, non-constant, smooth,
and positive function defined on $\mathbb{R}$, it does not exist  solution $u\in C^2(\mathbb{R}^+\times\mathbb{R}^d;\mathbb{R})$ to equation \eqref{periodic} for arbitrary small initial data.
This means that the oscillating coefficients
are responsible for the blow-up of the solutions.

 When the
time period is required to be a rational multiple of  the length of the spatial interval, (i.e., the  frequency $\omega$ is a rational number),  the pioneering work of ``periodic problems'' was given by Rabinowitz \cite{Rabinowitz1967periodic,rabinowitz1968periodic}, in which he established the existence of periodic solutions of forced dissipative wave equations subject to Dirichlet boundary conditions
\begin{align*}
&u_{tt}-u_{xx}+\alpha u_t=\epsilon f(t,x,u,u_x,u_t),\quad\alpha\neq0,\\
&u_{tt}-u_{xx}+\alpha u_t=\epsilon f(t,x,u,u_x,u_t,u_{xx},u_{xt},u_{tt}),\quad \alpha\neq0.
\end{align*}
 By the asymptotic expansions together with the contraction mapping theorem, Calleja et al.  \cite{calleja2017response} obtained  response solutions of several models of wave equations with quasi-periodic external forcing and very strong damping, such as the term $\frac{1}{\epsilon}u_t$ for $\epsilon$ small enough.   Another developments on response solutions for finite dimensional dynamical systems with damping and quasi-periodic forcing are for example in \cite{calleja2013construction,gentile2019forced}. In addition, we refer the readers to \cite{gauckler2016long} for wave equations with slowly varying wave speed
\begin{align*}
u_{tt}-a(\epsilon t)\Delta u=f(u),
\end{align*}
where  $\epsilon$ is  small enough. Gauckler et al. showed long-time near-conservation of the harmonic energy divided by the coefficient $a(\cdot)$  for such a model.  For  the frequency $\omega$ is rational multiple of $\pi$, Barbu--Pavel \cite{Barbu1996Periodic,Barbu1997Periodic} and  Ji--Li \cite{ji2006periodic,ji2011time} established the existence of periodic solutions to wave equations with space-dependent  coefficients.
If the  frequency $\omega$ is irrational, then it will arise  a ``small divisors problem''.
To overcome the ``small divisors problem'', a quite different approach
which used the KAM theory was developed from the
viewpoint of infinite dimensional dynamical systems by Kuksin \cite{kuksin1987hamiltonian} and Wayne
\cite{wayne1990periodic}. One can also look into \cite{Gao2009,Procesi2017KAM,Procesi2015,Eliasson2016beam} for
the application of this theory to construct periodic or quasi-periodic solutions to simi-linear PDEs related to ``small  divisors problem''.
Later, Craig--Wayne \cite{craig1993newton} and Bourgain  \cite{bourgain1994construction,bourgain1995construction} applied the Nash--Moser method to infinite-dimensional dynamical systems associated with ``small  divisors problem''. The advantage of this approach  is to require only the ``first order Melnikov'' non-resonance conditions, which are essentially the minimal assumptions. For this reason, their technique can overcome some limitations inherent to the usual KAM procedures.
Further developments on semi-linear PDEs can be found  in the articles \cite{berti2006cantor,berti2015abstract}. 

Different with the constant coefficient case ($a(t)\equiv a$) or space-dependent case ($a=a(x)$), this paper is devoted to the existence of families of small amplitude response solutions for equation \eqref{E1.1} with  the  irrational frequency. To the best of our knowledge,  there are no  result
available so far on the existence of periodic or quasi-periodic solutions of wave equations with
varying wave speed.
It turns our that the existence of periodic solution to the time-dependent coefficient wave equation is more challenging than that for the constant coefficient.
 We are faced with  some  difficulties:
\begin{itemize}
\item Arising of  a time-dependent principal operator $a( t)\partial_{xx}$,  where $a(t)$ can vary in any given bounded
interval that is bounded away from $0$.  If we do Fourier expansion as usual, it can not  appropriately  give the definition of the spectrum and the suitable  nonresonant condition. 
Furthermore,  it leads to the complexity of the  eigenvalues for the  following Hill problem
\begin{align*}
\textstyle -y''+\frac{\epsilon}{\omega^2}\Pi_{V}f'(v(\epsilon,\omega,w)+w)y=\lambda a y
\end{align*}
with respect to periodic boundary conditions.
We need to  establish the
exact relationship between the periodic coefficient $a(t)$ and the asymptotic formulae of the spectrum.
Since the influence of the asymptotic,  we have to give not only the invertibility in a appropriate space, but also the bound of the inverses of the ``truncated-restricted'' of the linearized operator associated with a range equation.
In addition, the arising  of $a(t)$ leads to that  the iteration and measure estimation are  more tedious, which
needs to give the detail analysis at every step (see Lemma \ref{lemma4.1} and  Lemma \ref{lemma4}).

\item Existence of resonant frequencies $-\omega^2\lambda^{\pm}_l+j^2$ for $ l\in\mathbb{N},j\in\mathbb{Z}$. The irrational frequency will cause the so-called ``small divisors problem''. Such a problem arises in the inversion of  the spectrum of the linear operator corresponding to equation \eqref{E1.1}, whose spectrum  has the   following form
\begin{align*}
\textstyle -\omega^2\lambda^{\pm}_l+j^2=-\omega^2\frac{l^2}{c^2}+j^2+b+O(\frac{1}{l}), \quad l\rightarrow+\infty,j\in\mathbb{Z}.
\end{align*}
Provided the assumption  $b\neq0$, the above term approaches to zero for almost every $\omega$. Generally, the inversion of the linear operator cannot map a function space into itself, but only into  a large  one with less regularity. This leads to unavailability of the classical implicit function theorem. For this reason, we need to impose the ``first order Melnikov'' non-resonance conditions, and then apply the Nash--Moser method.

\item  The invalidity of Floquet theory. Generally, it is impossible to find a transformation to convert  the wave equation  with periodic time-dependent coefficients into an equation
with constant coefficients. We can not define an appropriate  monodromy
operator (the shift along the trajectories of the system onto the period) due to the corresponding
Cauchy problem is ill-posed (see \cite{yagdjian2005global}).  It follows from the Floquet theory for PDEs  introduced by Kuchment \cite{Kuchment1982} that we cannot get the quasimomentums and Floquet exponents for our model \eqref{E1.1}.
Our proof relies on a  Lyapunov--Schmidt reduction together with the Nash--Moser iteration.
By imposing a periodic force term $g$,  the present article shows that equation \eqref{E1.1} admits families of small amplitude response solutions.  More precisely, under periodic boundary conditions and non-resonance
conditions on the frequency, when $\epsilon$ is small and $(\epsilon,\omega) $ belong to a Cantor set of positive measure, asymptotically full as $\epsilon$ tends to zero,  equation \eqref{E1.1} possesses small amplitude periodic solutions of the form $u=u(\omega t,x)$  in appropriate spaces consisted of smooth functions.

\end{itemize}

\subsection{Main result}\label{sec:result}
By the rescaling for temporal variable $t\rightarrow{t}/{\omega}$,  equation \eqref{E1.1} is equivalent to
\begin{align}\label{E1.2}
\omega^2u_{tt}-a(t) u_{xx}=\epsilon f(x,u)+\epsilon g(t,x),\quad x\in\mathbb{T}, t\in \mathbb{R}.
\end{align}
Our goal will be to look for $2\pi$-periodic solutions of the above equation.

For $s\geq0$, denote by $H^{s}$ the Sobolev spaces of real-valued functions
\begin{align}\label{E2.2}
H^{s}:=H^{s}(\mathbb{T}^2;\mathbb{R}):=\Big\{\textstyle u(t,x)=\sum_{j\in\mathbb{Z}}u_{j}(t)e^{{\rm i}j x}:u_j\in {H}^1(\mathbb{T};\mathbb{C}),u_{-j}=\overline{u_j}, \nonumber\\
\textstyle\|u\|^{2}_{s}:=\sum_{j\in\mathbb{Z}}\|u_j\|^2_{H^1}(1+ j^{2s})<+\infty\Big\},
\end{align}
where $\overline{u_j}$ is the complex conjugate of $u_j$. Moreover we introduce the space
\begin{align*}
\mathcal{C}_\ell:=\left\{f\in C(\mathbb{T}\times\mathbb{R};\mathbb{R}):(x,u)\longmapsto f(x,u)\text{ belongs to } C^\ell(\mathbb{T}\times\mathbb{R};H^1(\mathbb{T}))\right\}.
\end{align*}
For every $f\in\mathcal{C}_\ell$, if $f(x,u)=\sum_{j\in\mathbb{Z}}f_{j}(u)e^{{\rm i}j x}$ with $f_{-j}=\overline{f_j}$, then $u\longmapsto f_{j}(u)$ are in $C^{\ell}(\mathbb{R};H^{1}(\mathbb{T};\mathbb{C}))$.
\begin{rema}\label{remark1}
Let $C(s)$ be a positive constant depending on $s$. If $s>1/2$, then the space $H^s$ satisfies that for $ u,v\in H^s$,
\begin{align*}
(\mathrm{i})~\|uv\|_{{s}}\leq C(s)\|u\|_{s}\|v\|_{s}; \quad(\mathrm{ii})~\|u\|_{L^{\infty}(\mathbb{T};{H}^1(\mathbb{T}))}\leq C(s)\|u\|_{s}.
\end{align*}
\end{rema}
\begin{proof}
The proof is given in the Appendix.
\end{proof}
Let us denote
\begin{align*}
\textstyle V:={H}^1(\mathbb{T}),\quad W:=\left\{ w=\sum_{j\neq0}w_{j}(t)e^{\mathrm{i}j x}\in H^0\right\}.
\end{align*}
It is obvious that the space $H^s$ can be decomposed as the direct sum of $V$ and $W\cap H^s$, that is,
\begin{align*}
H^s=(V\cap H^s)\oplus (W\cap H^s)=V\oplus (W\cap H^s).
\end{align*}
Then the corresponding projection operators are defined by
\begin{align*}
\Pi_{V}:H^s\longrightarrow V,\quad\Pi_{W}:H^s\longrightarrow W.
\end{align*}
In addition, for every $u\in H^s$, we can write $u=v+w$, where $v\in V$ and $w\in W\cap H^s$. By implementing  the Lyapunov--Schmidt reduction with respect to the above decomposition, equation \eqref{E1.2} is equivalent to the bifurcation equation
\begin{align}
\omega^2v''=\epsilon\Pi_{V}F(v+w)+\epsilon\Pi_{V} g \label{E:bifurcation}
\end{align}
and the range equation
\begin{align}
L_{\omega}w=\epsilon\Pi_{W}F(v+w)+\epsilon\Pi_{W} g,\label{E:range}
\end{align}
where
\begin{align}\label{E1.8}
L_{\omega}w:=\omega^2w_{tt}-a(t)w_{xx},\quad F:u\longrightarrow f(x,u).
\end{align}
Observe that $f$ can be written into $f(x,u)=f_{0}(u)+\tilde{f}(x,u)$ with $\tilde{f}(x,u)=\sum_{j\neq0}f_{j}(u)e^{{\rm i}jx}$. For $w=0$, one derives
 \begin{align*}
\Pi_{V}F(v)=\Pi_{V}f(x,v)=\Pi_{V}f_{0}(v)+\Pi_{V}\tilde{f}(x,v)=f_{0}(v).
\end{align*}
Hence, if $w$ tends towards $0$, then the bifurcation equation is reduced to the space-independent equation
\begin{align*}
\omega^2v''=\epsilon f_{0}(v)+\epsilon g_{0}, \quad t\in \mathbb{T}.
\end{align*}
This is called the infinite-dimensional ``zeroth-order bifurcation equation'' introduced in \cite{berti2006cantor}. To fix ideas, we shall take $\omega$ inside an open sub-interval $(\mathcal{A}, \mathcal{B}) \subset (0, \infty)$ with $0<\mathcal{A}<\mathcal{B}<+\infty$. Without loss of generality, we consider $\omega\in (\frac{1}{2},\frac{3}{2})$. Let us suppose the following hypothesis.
\begin{assu}\label{hy1}
Let $f\in \mathcal{C}_\ell$ with $f(\cdot,0)=0$, and $x\longmapsto g(\cdot,x)$ be in $C^\ell(\mathbb{T};H^1(\mathbb{T}))$. For any $\omega\in (\frac{1}{2},\frac{3}{2})$, there exists a positive constant $\epsilon_0$  small enough such that for all $\epsilon\in[0,\epsilon_0]$, the problem
\begin{align}\label{E1.6}
\omega^2v''(t)=\epsilon f_{0}(v(t))+\epsilon g_0,\quad t\in\mathbb{T}
\end{align}
admits a non-degenerate solution $\hat{{v}}\in {H}^1(\mathbb{T})$ with $\int_{\mathbb{T}}\hat{v}(t)\mathrm{d}t=0$.
\end{assu}
The above hypothesis means that the linearized equation
\begin{align}\label{E1.7}
\omega^2h''=\epsilon f'_0(\hat{{v}})h, \quad h\in H^1(\mathbb{T}),
\end{align}
with $\int_{\mathbb{T}}h(t)\mathrm{d}t=0$, possesses only the trivial solution $h=0$ in $H^1(\mathbb{T})$. In fact, since $\int_{\mathbb{T}}h(t)\mathrm{d}t=0$, equation \eqref{E1.7} with $\epsilon=0$ possesses only the trivial solution $h=0$. Hence $\hat{v}=0$ is the non-degenerate solution  of \eqref{E1.6} with $\epsilon=0$. In view of the implicit function theorem, there exists $\epsilon_0>0$ such that for all $\epsilon\in[0,\epsilon_0]$, Hypothesis \ref{hy1} follows.

For fixed $\chi>1$, $\tau\in(1,2)$, and $\sigma=\frac{\tau(\tau-1)}{(2-\tau)}$, we denote
\begin{align}\label{E5.35}
\textstyle \beta:=\chi(\tau-1+\sigma)+\chi(2\tau+2)+\frac{\chi}{\chi-1}(\tau-1+\sigma).
\end{align}
The next theorem corresponds to a result on the existence of periodic solutions for the  wave equation subject to periodic external forcing.
\begin{theo}\label{Th1}
For fixed $\hat{\epsilon}\in[0,\epsilon_0]$, supposed that Hypothesis \ref{hy1} could hold.  Let $\tau\in(1,2)$, $\gamma\in(0,1)$, $\ell\geq s+\beta+3$, $a\in H^3(\mathbb{T})$ with $a(t)>0$, and the nonlinearity $f\in \mathcal{C}_\ell$ with $f(\cdot,0)=0$. If $x\longmapsto g(\cdot,x)$ is in $C^\ell(\mathbb{T};H^1(\mathbb{T}))$, then there exists a positive constant $\delta$ small enough, a neighborhood $(\epsilon_1,\epsilon_2)$ of $\hat{\epsilon}$, $0<r<1$, a $C^2$ mapping $v(\epsilon,\omega,w)$ on $(\epsilon_1,\epsilon_2)\times(\frac{1}{2},\frac{3}{2})\times\left\{w\in W\cap H^s:\|w\|_{s}< r\right\}$, with values in $V$, a mapping ${w}(\epsilon,\omega)\in W\cap H^s$,
and a Cantor set $A_{\gamma}\subset (0,\delta\gamma^2)\times(\frac{1}{2},\frac{3}{2})$, such that for all $(\epsilon,\omega)\in A_\gamma$,
\begin{align*}
{u}(\epsilon,\omega):=v(\epsilon,\omega,{w}(\epsilon,\omega))+{w}(\epsilon,\omega)\in V\oplus (W\cap H^s),
\end{align*}
with $\int_{\mathbb{T}^2}u(\epsilon,\omega)(t,x)\mathrm{d}t\mathrm{d}x=0$, is a solution of equation \eqref{E1.2} and satisfies that for all $x\in\mathbb{T}$, ${u}(\cdot,x)\in {H}^3(\mathbb{T})$.
\end{theo}

\subsection{Solutions of the bifurcation equation}\label{sec:birfucation}
In the following we will solve the bifurcation equation \eqref{E:bifurcation}, which corresponds to equation \eqref{E1.2}. The proof relies on the classical implicit function theorem under Hypothesis \ref{hy1}.
\begin{lemm}\label{lemma3.1}
Let Hypothesis \ref{hy1} hold for fixed $\hat{\epsilon}\in[0,\epsilon_0]$.  Then there is a neighborhood $(\epsilon_1,\epsilon_2)$ of $\hat{\epsilon}$, $0<r<1$, and a $C^2$ mapping
\begin{align*}
\textstyle v:(\epsilon_1,\epsilon_2)\times(\frac{1}{2},\frac{3}{2})\times\left\{w\in W\cap H^s:\|w\|_{s}< r\right\}\longrightarrow H^1(\mathbb{T}),\quad(\epsilon,\omega,w)\longmapsto v(\epsilon,\omega,w)
\end{align*}
such that $v(\epsilon,\omega,w)$ is a solution to equation \eqref{E:bifurcation} and satisfies $\int_{\mathbb{T}}v(\epsilon,\omega,w)(t)\mathrm{d}t=0$ and
\begin{align}\label{E2.14}
\| v(\epsilon,\omega,w)-v(\epsilon,\omega,0)\|_{H^1}\leq C\|w\|_s,\quad\|v(\epsilon,\omega,0)-\hat{v} \|_{H^1}\leq C(|\epsilon-\hat{\epsilon}|+|\omega-\hat{\omega}|).
\end{align}
\end{lemm}
\begin{proof}
If $x\longmapsto g(\cdot,x)$ is in $C^\ell(\mathbb{T};H^1(\mathbb{T}))$, from both the fact $f\in\mathcal{C}_\ell$ and Lemma \ref{lemma2.4}, then the mapping
\begin{align*}
(\epsilon,\omega,w,v)\longmapsto \omega^2v''-\epsilon\Pi_{V}F(v+w)-\epsilon\Pi_{V} g
\end{align*}
belongs to $C^2((\epsilon_1,\epsilon_2)\times(\frac{1}{2},\frac{3}{2})\times(W\cap H^s)\times H^1(\mathbb{T});H^1(\mathbb{T}))$. Moreover, according to Hypothesis \ref{hy1}, the linearized operator $h\longmapsto \omega^2h''-\hat{\epsilon}f'_0(\hat{v})h$, with $\int_{\mathbb{T}}h(t)\mathrm{d}t=0$, is invertible on $H^1(\mathbb{T})$. Hence the implicit function theorem shows that there is a $C^2$-path $(\epsilon,\omega,w)\longmapsto v(\epsilon,\omega,w)$, with $\int_{\mathbb{T}}v(\epsilon,\omega,w)(t)\mathrm{d}t=0$ and $v(\hat{\epsilon},\omega,0)=\hat{v}$, such that $v(\epsilon,\omega,w)$ is a solution of the bifurcation equation. In addition, it is clear that two terms in \eqref{E2.14} follow.
\end{proof}

If we set $\mathcal{F}(\epsilon,\omega,w):=F(v(\epsilon,\omega,w)+w)$, then the range
equation \eqref{E:range} can be rewritten as
\begin{align}\label{E:rang22}
L_{\omega}w=\epsilon \Pi_{W}\mathcal{F}(\epsilon,\omega,w)+\epsilon\Pi_{W} g.
\end{align}
Our task now is to prove the existence of solutions to equation \eqref{E:rang22} in the space $W\cap H^s$. The proof is based on the Nash--Moser iteration. Based on this method, the key is to investigate the invertibility of the corresponding linearized operator.

The  paper is structured as follows: We obtain the invertibility and the bound of the inverses of the ``truncated-restricted'' of the linearized operator with respect to the range equation in section \ref{sec:invertibility}. The aim of  section \ref{sec:range} is to construct
  solution of the rang equation. In section \ref{sec:proof}, we give the proof of the main result.

\section{Linearized operator}\label{sec:invertibility}
The purpose of the section is to show the invertibility  of the ``truncated-restricted'' of the linearized operator associated with  the range equation and give the bound of its inverses.

 Consider the orthogonal splitting $W=W_{N}\oplus W_{N}^{\bot}$, where
\begin{align*}
\textstyle W_{N}:=\left\{w\in W:w=\sum_{1\leq|j|\leq N}w_{j}(t)e^{\mathrm{i}j x}\right\},\quad W_{N}^{\bot}:=\left\{w\in W:w=\sum_{|j|> N}w_{j}(t)e^{\mathrm{i}j x}\right\}.
\end{align*}
For $N\in\mathbb{N}\backslash \{0\}$, the corresponding projectors
$\mathrm{P}_{N}:W\longrightarrow W_{N}, \mathrm{P}^{\bot}_{N}:W\longrightarrow W_{N}^{\bot}$ satisfy
\begin{align}
&\|\mathrm{P}_{N}w\|_{s+r}\leq N^r\|w\|_{s},\quad\forall w\in W\cap H^s,\label{E4.52}\\
&\|\mathrm{P}^{\bot}_{N}w\|_{s}\leq N^{-r}\|w\|_{{s+r}},\quad\forall w\in W\cap H^{s+r}.\label{E4.53}
\end{align}
Moreover, let us define the ``truncated-restricted'' of the linearized operator by
\begin{align}\label{E4.32}
\mathcal{L}_{N}(\epsilon,\omega,w)h:=-L_{\omega}h+\epsilon \mathrm{P}_{N}\Pi_{W}\mathrm{D}_{w}\mathcal{F}(\epsilon,\omega,w)[h],\quad
h\in W_{N},
\end{align}
where ${\rm D}_{w} \mathcal{F} $ is the Fr\'{e}chet derivative of $\mathcal{F}$ with respect to $w$, and $L_{\omega}$ is given by \eqref{E1.8}. Let $\lambda^{\pm}_{l}(\epsilon,\omega,w),l\geq0$ denote the eigenvalues for the Hill problem
\begin{align}\label{E6.1}
\begin{cases}
-y''+\frac{\epsilon}{\omega^2}\Pi_{V}f'(v(\epsilon,\omega,w)+w)y=\lambda a  y,\\
y(t)=y(t+2\pi),\quad y'(t)=y'(t+2\pi).
\end{cases}
\end{align}
Note that $\lambda_{0}(\epsilon,\omega,w)$ is denoted by $\lambda^{\pm}_{0}(\epsilon,\omega,w)$ for simplicity. Set
\begin{align}\label{E6.23}
\textstyle c:=\frac{1}{2\pi}\int_{\mathbb{T}}(a(t))^{\frac12}\mathrm{d}t.
\end{align}
For fixed $\gamma\in(0,1),\tau\in(1,2)$,  define
\begin{align*}
\Delta^{\gamma,\tau}_{N}(w):=\Big\{\textstyle(\epsilon,\omega)\in(\epsilon_1,\epsilon_2)\times(\frac{1}{2},\frac{3}{2}): {\epsilon}<\delta\gamma^2,~|\omega(\lambda^{\pm}_{l}(\epsilon,\omega,w))^{\frac{1}{2}}-j|>\frac{\gamma}{j^{\tau}},\nonumber\\
 \textstyle|\frac{\omega{l}}{c}-{j}|>\frac{\gamma}{j^{\tau}},\forall 1\leq j\leq N,\forall l\geq0\Big\}.
\end{align*}
Remark that
\begin{align*}
(\lambda^{\pm}_{l}(\epsilon,\omega,w))^{\frac{1}{2}}=
\mathrm{i}(-\lambda^{\pm}_{l}(\epsilon,\omega,w))^{\frac{1}{2}},\quad \text{if } \lambda^{\pm}_{l}(\epsilon,\omega,w)<0.
\end{align*}
The non-resonance conditions in the set $\Delta^{\gamma,\tau}_{N}(w)$ are trivially satisfied for  $\lambda^{\pm}_{l}(\epsilon,\omega,w)<0$.

\begin{lemm}\label{lem:invertibility}
For $(\epsilon,\omega)\in\Delta^{\gamma,\tau}_{N}(w)$, there exist positive constants $K, K(s')$ such that if
\begin{align}\label{E0.17}
\textstyle \|w\|_{{s+\sigma}}\leq1,\quad \text{with }\sigma=\frac{\tau(\tau-1)}{(2-\tau)},
\end{align}
 then the operator $\mathcal{L}_{N}(\epsilon,\omega,w)$ is invertible satisfying that for all $s'\geq s>1/2$,
\begin{align*}
&\|\mathcal{L}^{-1}_{N}(\epsilon,\omega,w)h\|_{s}\leq{K}{\gamma}^{-1}N^{\tau-1}\|h\|_{s},\\
&\|\mathcal{L}^{-1}_{N}(\epsilon,\omega,w)h\|_{{s'}}\leq {K(s')}{\gamma}^{-1}N^{\tau-1}\left(\|h\|_{{s'}}+\|w\|_{{s'+\sigma}}\|h\|_{s}\right).
\end{align*}
\end{lemm}
Lemma \ref{lem:invertibility} will be used in next section to prove Lemma \ref{lemma4.1}. The corresponding proof is divided into the following steps.

\subsection{Periodic boundary value problem}
We first propose the asymptotic formulae of  the spectrum associated with  the Hill problem \eqref{E6.1}. Denote
\begin{align}\label{E6.36}
\textstyle\mathrm{g}(t):=\frac{\epsilon}{\omega^2}\Pi_{V}f'(x,v(\epsilon,\omega,w(t,x))+w(t,x)),\text{ with }\mathrm{g}\in L^2(\mathbb{T}).
\end{align}
 Since $a(\cdot)$ is $2\pi$-periodic,
if $t=2\pi k+t'$ with $k\in\mathbb{Z}, t'\in[0,2\pi]$, by omitting the $'$ superscript of $t'$, then we just consider
\begin{align}\label{Hill}
\begin{cases}
-y''+\frac{\epsilon}{\omega^2}\Pi_{V}f'(v(\epsilon,\omega,w)+w)y=\lambda a  y,\\
y(0)=y(2\pi),\quad y'(0)=y'(2\pi).
\end{cases}
\end{align}
Let $c$ be as seen in \eqref{E6.23}. Making the Liouville substitution
\begin{align}\label{E6.26}
\textstyle t=\psi(\xi)\in[0,2\pi]\Longleftrightarrow\xi=\phi(t)\in[0,2\pi],\text{ with }\phi(t):=\frac{1}{c}\int^{t}_{0}(a(s))^{\frac12}\mathrm{d}s,
\end{align}
gives that $\lambda$ and $y(\psi(\xi))$ satisfy
\begin{align*}
\begin{cases}
-y''(\psi(\xi))-c\eta(\psi(\xi)) y'(\psi(\xi))+c^2\frac{\mathrm{g(\psi(\xi))}}{a (\psi(\xi))}y(\psi(\xi))=c^2\lambda y(\psi(\xi)),\\
y(0)=y(2\pi),\quad y'(0)=y'(2\pi),
\end{cases}
\end{align*}
where $\eta(t):=\frac{1}{2}a'(t)(a(t))^{-\frac32}$. 
We further make the Liouville transformation
\begin{align}\label{E6.35}
\textstyle y={z}/\upsilon\quad \text{ with }\upsilon(t)=(a(t))^{\frac{1}{4}}.
\end{align}
Then
\begin{align}\label{E6.42}
\begin{cases}
-z''(\xi)+\varrho(\xi)z(\xi)=\mu z(\xi),\\
z(0)=z(2\pi),\quad z'(0)=z'(2\pi),
\end{cases}
\end{align}
where
\begin{align}
&\textstyle \varrho(\xi)=q(\xi)+\theta(\xi),\label{E7.1}\\
&\textstyle q(\xi)=c^2Q(\psi(\xi)),\text{ with } Q(t)=\frac{a''(t)}{4(a(t))^2}-\frac{5(a'(t))^2}{16(a(t))^3},\label{E7.2}\\
&\textstyle \theta(\xi)=c^2\frac{\mathrm{g}(\psi(\xi))}{a(\psi(\xi))},\quad\mu=c^2\lambda.\label{E6.6}
\end{align}
The following Lemmata \ref{le:sequence}--\ref{lemma6.3} can be found  in the literature \cite{marchenko1977sturm,Levitan1991,titchmarsh1958eigenfunction}.
\begin{lemm}\label{le:sequence}
Denote by $\mu_0,\mu^{-}_{1},\mu^{+}_{1},\cdots,\mu^{-}_{l},\mu^{+}_{l},\cdots$ and  $\varphi_0$, $\varphi^{-}_1$, $\varphi^{+}_1,\cdots,\varphi^{-}_{l}$, $\varphi^{+}_{l},\cdots$, respectively, the eigenvalues and the corresponding real orthonormal eigenfunctions of  \eqref{E6.42}. Then the above eigenvalues are arranged as an increasing unbounded sequence $\mu_0<\mu^{-}_{1}\leq\mu^{+}_{1}<\cdots<\mu^{-}_{l}\leq\mu^{+}_{l}<\cdots$ with $\mu^{\pm}_{l}\rightarrow+\infty$ as $l\rightarrow+\infty$ so that if the equality sign is present, then the corresponding eigenvalue is double. In addition, the zeros of an eigenfunction on the segment $[0,2\pi)$ are equal to $2l$, where $l$ is the number of the corresponding eigenvalue.
\end{lemm}
In fact, formula \eqref{E7.1} shows that the eigenvalues $\mu^{\pm}_l$ of \eqref{E6.42} depend on $\theta(\cdot)$ (remark that $\mu_0$ is denoted by $\mu^\pm_0$ for convenience). However we do not write $\theta(\cdot)$.
Let us give the asymptotic formulae of the spectrum associated with  \eqref{E6.42}. Moreover, we denote $\varphi_{0}$ by $\varphi^{\pm}_{0}$ for brevity.
\begin{lemm}[Asymptotic formulae]\label{lemma6.3}
Let $\mu^{\pm}_0,\mu^{-}_{1},\mu^{+}_{1},\cdots,\mu^{-}_{l},\mu^{+}_{l},\cdots$ and  $\varphi^{\pm}_0$, $\varphi^{-}_1$, $\varphi^{+}_1,\cdots,\varphi^{-}_{l}$, $\varphi^{+}_{l},\cdots$ stand for the eigenvalues and the corresponding real orthonormal eigenfunctions of \eqref{E6.42}, respectively. Then the following asymptotic formulae hold:
\begin{align*}
\textstyle\mu^{\pm}_l=l^2+\frac{1}{2\pi}\int^{2\pi}_0\varrho(\xi)\mathrm{d}\xi\pm\frac{1}{2}\left(a^2_l+b^2_l\right)^{\frac{1}{2}}+O\left(\frac{1}{l}\right)
\end{align*}
as $l\rightarrow+\infty$, where
\begin{align*}
\textstyle a_l:=\frac{1}{\pi}\int^{2\pi}_0\varrho(\xi)\cos(2l\xi)\mathrm{d}\xi, \quad b_l:=\frac{1}{\pi}\int^{2\pi}_0\varrho(\xi)\sin(2l\xi)\mathrm{d}\xi.
\end{align*}
\end{lemm}
Denote by $\lambda_{0}(\epsilon,\omega, w)$, $\lambda^{-}_{1}(\epsilon,\omega, w)$, $\lambda^{+}_{1}(\epsilon,\omega, w)$, $\cdots$,  $\lambda^{-}_{l}(\epsilon,\omega, w)$, $\lambda^{+}_{l}(\epsilon,\omega, w)$, $\cdots$ and $\psi_{0}(\epsilon,\omega, w)$, $\psi^{-}_{1}(\epsilon,\omega, w)$, $\psi^{+}_{1}(\epsilon,\omega, w)$, $\cdots$, $\psi^{-}_{l}(\epsilon,\omega, w)$, $\psi^{+}_{l}(\epsilon,\omega, w)$, $\cdots$ the eigenvalues and the corresponding real orthonormal eigenfunctions of the Hill problem \eqref{E6.1}, respectively. For brevity, we denote $\lambda_{0}(\epsilon,\omega, w)$ and $\psi_0(\epsilon,\omega, w)$ by $\lambda^{\pm}_{0}(\epsilon,\omega, w)$ and $\psi^{\pm}_0(\epsilon,\omega, w)$, respectively.
\begin{lemm}\label{lemma6.1}
Let $c$ be as seen in \eqref{E6.23}. For all $\epsilon\in(\epsilon_1,\epsilon_2), \omega\in(\frac{1}{2},\frac{3}{2}),w\in\{W\cap H^s:\|w\|_{s}< r\}$,  the eigenvalues of the Hill problem \eqref{E6.1} satisfy
\begin{align}\label{E0.1}
\lambda^{\pm}_{0}(\epsilon,\omega,w)&<\lambda^{-}_{1}(\epsilon,\omega,w)\leq\lambda^{+}_{1}(\epsilon,\omega,w)<\cdots<\lambda^{-}_{l}(\epsilon,\omega,w)\leq\lambda^{+}_{l}(\epsilon,\omega,w) < \cdots
\end{align}
with $\lambda^{\pm}_{l}(\epsilon,\omega,w)\rightarrow+\infty$ as $l\rightarrow+\infty$, and
\begin{align}\label{E6.24}
\lambda^{\pm}_l(\epsilon,\omega,w)=&\textstyle \frac{l^2}{c^2}+\frac{1}{2\pi c}\int_{\mathbb{T}}\textstyle\Gamma(\epsilon,\omega,w)( t)(a(t))^{\frac12}\mathrm{d}t+O\left(\frac{1}{l}\right)
\end{align}
as $l\rightarrow+\infty$, with
\begin{align*}
\textstyle\Gamma(\epsilon,\omega,w)(t):=Q(t)+\frac{\epsilon}{\omega^2a(t)}\Pi_{V}f'(x,v(\epsilon,\omega,w(t,x))+w(t,x)),
\end{align*}
where $Q$ is given by \eqref{E7.2}.

The eigenfunctions $\psi^{\pm}_{l}(\epsilon,\omega,w)$ form an orthogonal basis of $L^2(\mathbb{T})$ subject to the scalar product
\begin{align*}
\textstyle(y,\mathfrak{z})_{L^2_{a}}:=\frac{1}{c}\int_{\mathbb{T}}a y\mathfrak{z}\mathrm{d}t.
\end{align*}
Moreover, for $\Theta>0$ large enough, we can define an equivalent scalar product $(\cdot,\cdot)_{\epsilon,\omega,w}$ on ${H}^1(\mathbb{T})$ by
\begin{align*}
\textstyle(y,\mathfrak{z})_{\epsilon,\omega,w}:=\frac{1}{c}\int_{\mathbb{T}}\left(y'\mathfrak{z}'+\frac{\epsilon}{\omega^2}\Pi_{V}f'(v(\epsilon,\omega,w)+w)y\mathfrak{z}+\Theta ay\mathfrak{z}\right)\mathrm{d}t
\end{align*}
with, for  $L_1>0,L_2>0$,
\begin{align}\label{E6.3}
L_1\|y\|_{H^1}\leq\|y\|_{\epsilon,\omega,w}\leq L_2\|y\|_{H^1},\quad y\in{H}^1(\mathbb{T}).
\end{align}
Then $\psi^{\pm}_{l}(\epsilon,\omega,w)$ are an orthogonal basis of ${H}^1(\mathbb{T})$
 with respect to the scalar product $(\cdot,\cdot)_{\epsilon,\omega,w}$ as well. Furthermore, one has that for $ y=\hat{y}^{\pm}_0\psi^{\pm}_{0}(\epsilon,\omega,w)+\sum_{l\geq1}\hat{y}^{+}_l\psi^{+}_{l}(\epsilon,\omega,w)+\hat{y}^{-}_l\psi^{-}_{l}(\epsilon,\omega,w)$,
\begin{align}\label{E6.4}
&\textstyle\|y\|^2_{L^2_{a}}=(\hat{y}^{\pm}_0)^2+\sum_{l\geq1}(\hat{y}^{+}_l)^2+(\hat{y}^{-}_l)^2,\\
&\textstyle\|y\|^2_{\epsilon,\omega,w}=(\lambda^{\pm}_0(\epsilon,\omega,w)+\Theta)(\hat{y}^{\pm}_0)^2+\sum_{l\geq1}(\lambda^{+}_l(\epsilon,\omega,w)+\Theta)(\hat{y}^{+}_l)^2+(\lambda^{-}_l(\epsilon,\omega,w)+\Theta)(\hat{y}^{-}_l)^2,\label{scalar}
\end{align}
where $\hat{y}^{\pm}_0$ stands for $\hat{y}_0$ for simplicity and $\lambda^{\pm}_l(\epsilon,\omega,w)+\Theta>0$.
\end{lemm}
\begin{proof}
From \eqref{E6.6}, it is easy to see that $\mu^{\pm}_l(\theta)=c^2\lambda^{\pm}_l(\mathrm{g})$. In view of Lemma \ref{le:sequence}, we can derive \eqref{E0.1}. Moreover, by Lemma \ref{lemma6.3} and the inverse Liouville substitution in \eqref{E6.26}, the eigenvalues of \eqref{Hill} satisfy
\begin{align*}
\lambda^{\pm}_l(\epsilon,\omega,w)=&\textstyle \frac{l^2}{c^2}+\frac{1}{2\pi c}\int^{2\pi}_0\textstyle\Gamma(\epsilon,\omega,w)( t)(a(t))^{\frac12}\mathrm{d}t\pm\textstyle\frac{1}{2}\left(\Lambda^2_l(\epsilon,\omega,w)+\Upsilon^2_l(\epsilon,\omega,w)\right)^{\frac{1}{2}}+O\left(\frac{1}{l}\right)
\end{align*}
as $l\rightarrow+\infty$, where
\begin{align*}
&\textstyle \Lambda_l(\epsilon,\omega,w):=\frac{1}{\pi c}\int^{2\pi}_0\Gamma(\epsilon,\omega,w)( t)\cos\left(2l\left(\frac{1}{c}\int^{t}_{0}(a(s))^{\frac12}\mathrm{d}s\right)\right)(a(t))^{\frac12}\mathrm{d}t, \\
&\textstyle \Upsilon_l(\epsilon,\omega,w):=\frac{1}{\pi c}\int^{2\pi}_0\Gamma(\epsilon,\omega,w)( t)\sin\left(2l\left(\frac{1}{c}\int^{t}_{0}(a(s))^{\frac12}\mathrm{d}s\right)\right)(a(t))^{\frac12}\mathrm{d}t.
\end{align*}
Since
\begin{align*}
\textstyle|\Lambda_l(\epsilon,\omega,w)|\leq\frac{\|\Gamma(\epsilon,\omega,w)\|_{H^1}}{l},\quad |\Upsilon_l(\epsilon,\omega,w)|\leq\frac{\|\Gamma(\epsilon,\omega,w)\|_{H^1}}{l},
\end{align*}
it can be seen that \eqref{E6.24} holds.

Observe that the eigenfunctions $\varphi^{\pm}_{l}(\epsilon,\omega,w)$ (recall Lemma \ref{le:sequence}) form an orthonormal basis for $L^2(0,2\pi)$. Then it follows from the Liouville substitution \eqref{E6.35} that
\begin{align*}
\varphi^{\pm}_{l}(\epsilon,\omega,w)=\psi^{\pm}_l(\epsilon,\omega,w)\upsilon,\text{ with }\upsilon(t)=(a(t))^{\frac{1}{4}}.
\end{align*}
Combining this with the  inverse Liouville substitution in \eqref{E6.26} gives that
\begin{align*}
\textstyle\int_{0}^{2\pi}\varphi^{\pm}_{l}(\epsilon,\omega,w)(\xi)\varphi^{\pm}_{l}(\epsilon,\omega,w)(\xi)\mathrm{d}\xi=\textstyle\frac{1}{c}\int_{0}^{2\pi}\psi^{\pm}_l(\epsilon,\omega,w)(t)\psi^{\pm}_l(\epsilon,\omega,w)(t)a(t)\mathrm{d}t.
\end{align*}
Hence $\psi^{\pm}_{l}(\epsilon,\omega,w)$ form an orthogonal basis for $L^2(0,2\pi)$ with respect to the scalar product $(\cdot,\cdot)_{L^2_a}$. In addition, a simple calculation can read that \eqref{E6.3}--\eqref{E6.4} follow. Notice that
\begin{align*}
\textstyle-(\psi^{+}_l(\epsilon,\omega,w))''+\frac{\epsilon}{\omega^2}\Pi_{V}f'(v(\epsilon,\omega,w)+w)\psi^{+}_l(\epsilon,\omega,w)
+\Theta a\psi^{+}_l(\epsilon,\omega,w)&\\=(\lambda^{+}_l(\epsilon,\omega,w)+\Theta)a\psi^{+}_l(\epsilon,\omega,w)&,\\
\textstyle-(\psi^{-}_l(\epsilon,\omega,w))''+\frac{\epsilon}{\omega^2}\Pi_{V}f'(v(\epsilon,\omega,w)+w)\psi^{-}_l(\epsilon,\omega,w)
+\Theta a\psi^{-}_l(\epsilon,\omega,w)&\\
=(\lambda^{-}_l(\epsilon,\omega,w)+\Theta)a\psi^{-}_l(\epsilon,\omega,w)&,
\end{align*}
 Multiplying the above equality  by $\psi^{\pm}_{l'}(\epsilon,\omega,w)$ and integrating by parts yield that 
\begin{align*}
&(\psi^{+}_l,\psi^{+}_{l'})_{\epsilon,\omega,w}=\delta_{l,l'}(\lambda^{+}_l(\epsilon,\omega,w)+\Theta),\quad(\psi^{-}_l,\psi^{-}_{l'})_{\epsilon,\omega,w}=\delta_{l,l'}(\lambda^{-}_l(\epsilon,\omega,w)+\Theta),\\
&(\psi^{+}_l,\psi^{-}_{l'})_{\epsilon,\omega,w}=(\psi^{-}_l,\psi^{+}_{l'})_{\epsilon,\omega,w}=0\quad \text{if\quad either $l\neq0$ or $l'\neq0$}.
\end{align*}
Thus formula \eqref{scalar} holds. We now complete the proof of the lemma.
\end{proof}

Furthermore we need to introduce the following lemma.
\begin{lemm}
For all $\epsilon',\epsilon''\in(\epsilon_1,\epsilon_2),\omega',\omega''\in(\frac{1}{2},\frac32)$ and $w',w''\in W\cap H^s$, there is some positive constant $\kappa$ such that  the eigenvalues of \eqref{E6.1} satisfy
\begin{align}\label{E:continue}
\textstyle|\lambda^{\pm}_{l}(\epsilon',\omega',w')-\lambda^{\pm}_{l}(\epsilon'',\omega'',w'')|\leq
\kappa(|\epsilon'-\epsilon''|+|\omega'-\omega''|+\|w'-w''\|_s),\quad l\geq0.
\end{align}
In particular, for fixed $\epsilon\in(\epsilon_1,\epsilon_2)$, one has
\begin{align}\label{E:continue2}
\textstyle|\lambda^{\pm}_{l}(\epsilon,\omega',w')-\lambda^{\pm}_{l}(\epsilon,\omega'',w'')|\leq
\epsilon\kappa(|\omega'-\omega''|+\|w'-w''\|_s),\quad l\geq0.
\end{align}
\end{lemm}
\begin{proof}
Observe that $\mathrm g\in {H}^1(\mathbb T)\subseteq C(\mathbb T)$, where $\mathrm g$ is given in \eqref{E6.36}. For $l\in\mathbb{N}$, let $\psi^{\pm}_{l}(\mathrm g)$ denote the eigenfunctions of the  Hill problem \eqref{E6.1} with respect to $\lambda^{\pm}_{l}(\mathrm g)$. Since the coefficients $\mathrm{g}, a$ in \eqref{E6.1} satisfy the assumptions of Theorem 4.2 in \cite{zettl1996eigenvalues}, we obtain
\begin{align*}
\textstyle \mathrm{D}_{\mathrm g}\lambda^{\pm}_{l}(\mathrm g)[h]=\int_{\mathbb{T}}(\psi^{\pm}_{l}(\mathrm g))^2 h \mathrm{d}t.
\end{align*}
Then it follows from Lemma  \ref{lemma2.4}, Lemma \ref{lemma3.1} and Lemma \ref{lemma6.1} that
\begin{align*}
\textstyle|\lambda^{\pm}_{l}(\mathrm{g}_1)-\lambda^{\pm}_{l}(\mathrm{g}_2)|=&\textstyle|\int^{1}_{0}\int_{\mathbb{T}}
(\psi^{\pm}_{l}(\mathrm{g}_1+\mathfrak{v}(\mathrm{g}_2-\mathrm{g}_1)))^2(\mathrm{g}_1-\mathrm{g}_2) \mathrm{d}t\mathrm{d}\mathfrak{v}|\\
\leq&\textstyle\max_{\mathfrak{v}\in[0,1]}\left|\int_{\mathbb{T}}(\psi^{\pm}_{l}(\mathrm{g}_1+\mathfrak{v}(\mathrm{g}_2-\mathrm{g}_1)))^2
(\mathrm{g}_1-\mathrm{g}_2) \mathrm{d}t\right|\\
\leq&\textstyle\|\frac{c(\mathrm{g}_1-\mathrm{g}_2)}{a}\|_{L^{\infty}(\mathbb{T})}\max_{\mathfrak{v}\in[0,1]}
|\frac{1}{c}\int_{\mathbb{T}}(\psi^{\pm}_{l}(\mathrm{g}_1+\mathfrak{v}(\mathrm{g}_2-\mathrm{g}_1)))^2 a\mathrm{d}t| \\
\leq&\textstyle C\|\mathrm{g}_1-\mathrm{g}_2\|_{H^1(\mathbb T)}{\leq}\kappa(|\epsilon'-\epsilon''|+|\omega'-\omega''|+\|w'-w''\|_s),
\end{align*}
where
\begin{align*}
\textstyle\mathrm{g}_1=\frac{\epsilon'}{{\omega'}^2}\Pi_{V}f'(v(\epsilon',\omega',w')+w'),\quad \mathrm{g}_2=\frac{\epsilon''}{{\omega''}^2}\Pi_{V}f'(v(\epsilon'',\omega'',w'')+w'').
\end{align*}
In particular, for $\epsilon'=\epsilon''=\epsilon$, it can be seen that \eqref{E:continue2} holds.
We have thus obtained the conclusion of the lemma.
\end{proof}

By Hypothesis \ref{hy1}, the non-degeneracy of $\hat{v}=v(\hat{\epsilon},\omega,0)$ means that $\lambda^{\pm}_{l}(\hat{\epsilon},\omega, 0)\neq0,\forall\omega\in(\frac{1}{2},\frac32)$.
Then it follows from \eqref{E:continue} that
\begin{align}\label{E:min}
\textstyle\Gamma_0:=\inf\left\{|\lambda^{\pm}_l(\epsilon,\omega,w)|:~l\geq0,\epsilon\in[\epsilon_1,\epsilon_2],\omega\in[\frac{1}{2},\frac{3}{2}],\|w\|_s\leq r\right\}>0.
\end{align}
If necessary, here we may take that $|\epsilon_2-\epsilon_1|$ and $r$ are smaller than ones  in Lemma \ref{lemma3.1}.

Moreover, due to \eqref{E6.3} and \eqref{scalar} in Lemma \ref{lemma6.1}, we present that the $s$-norm restricted to $W\cap H^s$ has an equivalent norm. More precisely, if we decompose $w$ as
\begin{align*}
\textstyle\sum_{j}\hat{w}^{\pm}_{j,0}\psi^{\pm}_{0}(\epsilon,\omega,w)e^{\mathrm{i}jx}+\sum_{l\geq1,j}(\hat{w}^{+}_{j,l}\psi^{+}_{l}(\epsilon,\omega,w)+\hat{w}^{-}_{j,l}\psi^{-}_{l}(\epsilon,\omega,w))e^{\mathrm{i}jx},
\end{align*}
then
\begin{align}\label{E5.26}
\textstyle L^2_1\|w\|^2_{s}&\leq\textstyle\sum_{j\neq0}(\lambda^{\pm}_{0}(\epsilon,\omega,w)+\Theta)(\hat{w}^{\pm}_{j,0})^2(1+j^{2s})\nonumber\\
&+\textstyle\sum_{j\neq0,l\geq1}((\lambda^{+}_{l}(\epsilon,\omega,w)+\Theta)(\hat{w}^{+}_{j,l})^2+(\lambda^{-}_{l}(\epsilon,\omega,w)+\Theta)(\hat{w}^{-}_{j,l})^2)(1+j^{2s})\leq L^2_2\|w\|^2_{s}.
\end{align}
where $\hat{w}_{j,0}$ is denoted by $\hat{w}^{\pm}_{j,0}$ for brevity. Our task now is to investigate the invertibility and the bound of the inverses of the ``truncated-restricted'' of the linearized operator subject to equation \eqref{E:rang22}.

\subsection{Invertibility of the linearized operator}

We can write  the ``truncated-restricted'' of the linearized operator as
\begin{align*}
\mathcal{L}_{N}(\epsilon,\omega,w)h:=&-L_{\omega}h+\epsilon \mathrm{P}_N\Pi_{W}\mathrm{D}_{w}F(v(\epsilon,\omega,w)+w)[h]=\mathfrak{L}_1h+\mathfrak{L}_2h,\quad h\in W_{N},
\end{align*}
where $L_{\omega}$ is as seen in \eqref{E1.8}, and
\begin{align}
&\mathfrak{L}_1h:=-L_{\omega}h+\epsilon \mathrm{P}_N\Pi_Wf'(v(\epsilon,\omega,w)+w)h,\nonumber\\
&\mathfrak{L}_2h:=\epsilon \mathrm{P}_N\Pi_Wf'(v(\epsilon,\omega,w)+w)\mathrm{D}_{w}v(\epsilon,\omega,w)[h].\nonumber
\end{align}
Let  $b(t,x):=f'(x,v(\epsilon,\omega,w(t,x))+w(t,x))$. For $\|w\|_{{s+\sigma}}\leq1$, by using \eqref{E2.5} and Lemma \ref{lemma3.1}, it is easy to see that for all $s'\geq s>1/2$,
\begin{align}
&\| b\|_{s}\leq\|b\|_{{s+\sigma}}\leq C,\label{E5.15}\\
&\|b\|_{s'}\leq C(s')(1+\|w\|_{{s'}}).\label{E5.14}
\end{align}
If we express $b,h$ as  $\sum_{j\in\mathbb{Z}} b_{j}e^{\mathrm{i}j x},\sum_{1\leq |j|\leq N}h_je^{\mathrm{i}j x}$, respectively, it is easy to obtain that
\begin{align*}
\mathfrak{L}_1h&=\sum_{{1\leq|j|\leq N}}\left(-\omega^2\partial_{tt}(h_{j})-aj^2h_j\right)e^{\mathrm{i}j x}+\epsilon \mathrm{P}_{N}\Pi_W\left(\sum_{k\in\mathbb{Z},{1\leq|j|\leq N}}b_{k-j}h_je^{\mathrm{i}k x}\right)\nonumber\\
&=a\mathfrak{L}_{{1,\mathrm{D}}}h-a\mathfrak{L}_{{1,\mathrm{ND}}}h,
\end{align*}
where
\begin{align}
&\mathfrak{L}_{{1,\mathrm{D}}}h:=\textstyle\sum_{1\leq|j|\leq N}\left(-\frac{1}{a}\omega^2\partial_{tt}( h_{j})-j^2h_j+\epsilon\frac{b_0}{a} h_j\right)e^{\mathrm{i}j x},\nonumber\\
&\mathfrak{L}_{{1,\mathrm{ND}}}h:=-\textstyle\frac{\epsilon}{a}\sum_{1\leq|j|,|k|\leq N,j\neq k}b_{k-j}h_je^{\mathrm{i}k x},\label{E5.27}
\end{align}
with  $b_0(t)=\Pi_{V}f'(x,v(\epsilon,\omega,w(t,x))+w(t,x))$. Moreover, Lemma \ref{lemma6.1} shows that
\begin{align*}
\textstyle-\frac{1}{a}\omega^2h''_{j}-j^2h_j+\epsilon\frac{b_0}{a} h_j=&\textstyle(\omega^2\lambda^{\pm}_{0}(\epsilon,\omega,w)-j^2)\hat{h}^{\pm}_{j,0}\psi^{\pm}_0(\epsilon,\omega,w)\\
&\textstyle+\sum_{l\geq1}(\omega^2\lambda^{+}_{l}(\epsilon,\omega,w)-j^2)\hat{h}^{+}_{j,l}\psi^{+}_l(\epsilon,\omega,w)\\
&\textstyle+\sum_{l\geq1}(\omega^2\lambda^{-}_{l}(\epsilon,\omega,w)-j^2)\hat{h}^{-}_{j,l}\psi^{-}_l(\epsilon,\omega,w)
\end{align*}
for $h_{j}=\hat{h}^{\pm}_{j,0}\psi^{\pm}_0(\epsilon,\omega,w)+\sum_{l\geq1}\hat{h}^{+}_{j,l}\psi^{+}_l(\epsilon,\omega,w)+\hat{h}^{-}_{j,l}\psi^{-}_l(\epsilon,\omega,w)$, where we denote $\hat{h}_{j,0}$ by $\hat{h}^{\pm}_{j,0}$ for convenience. Then $\mathfrak{L}_{{1,\mathrm{D}}}$  is a diagonal operator on $W_{N}$. Let us define the operator
\begin{align*}
|\mathfrak{L}_{{1,\mathrm{D}}}|^{\frac{1}{2}}h=&\textstyle\sum_{1\leq|j|\leq N}|\omega^2\lambda^{\pm}_{0}(\epsilon,\omega,w)-j^2|^{\frac12} \hat{h}^{\pm}_{j,0}\psi^{\pm}_0(\epsilon,\omega,w)e^{\mathrm{i}j x} \\
&+\textstyle\sum_{1\leq|j|\leq N,l\geq1}|\omega^2\lambda^{+}_{l}(\epsilon,\omega,w)-j^2|^{\frac12} \hat{h}^{+}_{j,l}\psi^{+}_l(\epsilon,\omega,w)e^{\mathrm{i}j x} \\
&+\textstyle\sum_{1\leq|j|\leq N,l\geq1}|\omega^2\lambda^{-}_{l}(\epsilon,\omega,w)-j^2|^{\frac12}\hat{h}^{-}_{j,l}\psi^{-}_l(\epsilon,\omega,w)e^{\mathrm{i}j x},\quad h\in W_{N}.
\end{align*}
If  $\omega^2\lambda^{\pm}_{l}(\epsilon,\omega,w)-j^2\neq0$, $\forall1\leq|j|\leq N,~\forall l\geq0$, then its inverse operator is
\begin{align*}
\textstyle|\mathfrak{L}_{{1,\mathrm{D}}}|^{-\frac{1}{2}}{h}:=&\textstyle\sum_{1\leq|j|\leq N}|\omega^2\lambda^{\pm}_{0}(\epsilon,\omega,w)-j^2|^{-\frac12}\hat{h}^{\pm}_{j,0}
\psi^{\pm}_0(\epsilon,\omega,w)e^{\mathrm{i}j x}\\
&+\textstyle\sum_{1\leq|j|\leq N,l\geq1}|\omega^2\lambda^{+}_{l}(\epsilon,\omega,w)-j^2|^{-\frac12}\hat{{h}}^{+}_{j,l}
\psi^{+}_l(\epsilon,\omega,w)e^{\mathrm{i}j x}\\
&+\textstyle\sum_{1\leq|j|\leq N,l\geq1}|\omega^2\lambda^{-}_{l}(\epsilon,\omega,w)-j^2|^{-\frac12}\hat{{h}}^{-}_{j,l}
\psi^{-}_l(\epsilon,\omega,w)e^{\mathrm{i}j x},\quad h\in W_{N}.
\end{align*}
Hence the operator $\mathcal{L}_{N}(\epsilon,\omega,w)$  is reduced to
\begin{align*}
\mathcal{L}_{N}(\epsilon,\omega,w)=a|\mathfrak{L}_{1,{\mathrm{D}}}|^{\frac{1}{2}}(|\mathfrak{L}_{1,{\mathrm{D}}}|^{-\frac{1}{2}}
\mathfrak{L}_{1,{\mathrm{D}}}|\mathfrak{L}_{1,{\mathrm{D}}}|^{-\frac{1}{2}}-{R}_1-{R}_2)|\mathfrak{L}_{1,{\mathrm{D}}}|^{\frac{1}{2}},
\end{align*}
where
\begin{align}\label{E5.28}
\textstyle{R}_1=|\mathfrak{L}_{1,{\mathrm{D}}}|^{-\frac{1}{2}}
\mathfrak{L}_{1,{\mathrm{ND}}}|\mathfrak{L}_{1,{\mathrm{D}}}|^{-\frac{1}{2}},\quad
{R}_2=-|\mathfrak{L}_{1,{\mathrm{D}}}|^{-\frac{1}{2}}
\left(\frac{1}{a}\mathfrak{L}_2\right)|\mathfrak{L}_{1,{\mathrm{D}}}|^{-\frac{1}{2}}.
\end{align}
Moreover, the definitions of $\mathfrak{L}_{1,\mathrm{D}}$, $|\mathfrak{L}_{1,{\mathrm{D}}}|^{-\frac{1}{2}}$ read that for ${h}\in{W}_{N}$,
\begin{align*}
(|\mathfrak{L}_{1,{\mathrm{D}}}|^{-\frac{1}{2}}
\mathfrak{L}_{1,{\mathrm{D}}}|\mathfrak{L}_{1,{\mathrm{D}}}|^{-\frac{1}{2}}){h}=&\textstyle\sum_{1\leq|j|\leq N}\mathrm{sign}(\omega^2\lambda^{\pm}_{0}(\epsilon,\omega,w)-j^2)
\hat{h}^{\pm}_{j,0}\psi^{\pm}_{0}(\epsilon,\omega,w)e^{\mathrm{i}j x}\\
&+\textstyle\sum_{1\leq|j|\leq N,l\geq1}\mathrm{sign}(\omega^2\lambda^{+}_{l}(\epsilon,\omega,w)-j^2)
\hat{h}^{+}_{j,l}\psi^+_{l}(\epsilon,\omega,w)e^{\mathrm{i}j x}\\
&\textstyle+\sum_{1\leq|j|\leq N,l\geq1}\mathrm{sign}(\omega^2\lambda^{-}_{l}(\epsilon,\omega,w)-j^2)
\hat{h}^{-}_{j,l}\psi^-_{l}(\epsilon,\omega,w)e^{\mathrm{i}j x}.
\end{align*}
Combining this with \eqref{E5.26} implies that it is invertible with
\begin{align}
\textstyle\|(|\mathfrak{L}_{1,{\mathrm{D}}}|^{-\frac{1}{2}}
\mathfrak{L}_{1,{\mathrm{D}}}|\mathfrak{L}_{1,{\mathrm{D}}}|^{-\frac{1}{2}})^{-1}{h}\|_{s}
{\leq} \frac{L_2}{L_1}\|h\|_{s},\quad\forall s\geq0.\label{E5.30}
\end{align}
Finally, we write $\mathcal{L}_{N}(\epsilon,\omega,w)$ as
\begin{align}\label{L:equivalent}
\mathcal{L}_{N}(\epsilon,\omega,w)=a|\mathfrak{L}_{1,{\mathrm{D}}}|^{\frac{1}{2}}(|\mathfrak{L}_{1,{\mathrm{D}}}|^{-\frac{1}{2}}
\mathfrak{L}_{1,{\mathrm{D}}}|\mathfrak{L}_{1,{\mathrm{D}}}|^{-\frac{1}{2}})(\mathrm{Id}-\mathcal{R})
|\mathfrak{L}_{1,{\mathrm{D}}}|^{\frac{1}{2}},
\end{align}
where $\mathcal{R}=\mathcal{R}_1+\mathcal{R}_2$ with
\begin{align*}
\mathcal{R}_1=(|\mathfrak{L}_{1,{\mathrm{D}}}|^{-\frac{1}{2}}
\mathfrak{L}_{1,{\mathrm{D}}}|\mathfrak{L}_{1,{\mathrm{D}}}|^{-\frac{1}{2}})^{-1}R_1,\quad
\mathcal{R}_2=(|\mathfrak{L}_{1,{\mathrm{D}}}|^{-\frac{1}{2}}
\mathfrak{L}_{1,{\mathrm{D}}}|\mathfrak{L}_{1,{\mathrm{D}}}|^{-\frac{1}{2}})^{-1}R_2.
\end{align*}
To verify the invertibility of the operator $\mathrm{Id}-\mathcal{R}$, we have to impose some non-resonance conditions.

For $\tau\in(1,2),\gamma\in(0,1)$, we assume  the following ``Melnikov's'' non-resonance conditions
\begin{align}\label{E5.12}
|\omega(\lambda^{\pm}_{l}(\epsilon,\omega,w))^{\frac12}-j|>\frac{\gamma}{j^{\tau}},\quad\forall 1\leq j\leq N,~\forall l\geq 0.
\end{align}
As a result, it follows that
\begin{align}\label{E5.13}
\textstyle|\omega^2\lambda^{\pm}_{l}(\epsilon,\omega,w)-j^2|=|\omega(\lambda^{\pm}_{l}(\epsilon,\omega,w))^{\frac12}-j|
|\omega(\lambda^{\pm}_{l}(\epsilon,\omega,w))^{\frac12}+j|>\frac{\gamma}{j^{\tau-1}}.
\end{align}
Furthermore, denote
\begin{align}\label{E6.27}
\omega^{\pm}_{j}:=\min_{l\geq0}|\omega^2\lambda^{\pm}_{l}(\epsilon,\omega,w)-j^2|=|\omega^2\lambda^{\pm}_{l^{*}}(\epsilon,\omega,w)-j^2|, \quad1\leq|j|\leq N.
\end{align}
It is clear that $\omega^{\pm}_j=\omega^{\pm}_{-j}$ for all $1\leq j\leq N$.

\begin{lemm}
If  conditions \eqref{E5.12} are given, then the operator $|\mathfrak{L}_{{1,\mathrm{D}}}|^{-\frac{1}{2}}$
satisfies that for ${h}\in{W}_{N}$,
\begin{align}\label{E5.17}
&\textstyle\||\mathfrak{L}_{{1,\mathrm{D}}}|^{-\frac{1}{2}}{h}\|_{{s}}\leq\frac{\sqrt{2}L_2}{\sqrt{\gamma}L_1}\|{h}\|_{{s+\frac{\tau-1}{2}}},\quad\forall s\geq0,\\
&\textstyle\||\mathfrak{L}_{{1,\mathrm{D}}}|^{-\frac{1}{2}}{h}\|_{{s}}\leq\frac{\sqrt{2}{L_2}}{\sqrt{\gamma}L_1}N^{\frac{\tau-1}{2}}\|{h}\|_{{s}} ,\quad\forall s\geq0.\label{E5.17.2}
\end{align}
\end{lemm}
\begin{proof}
Observe that $|j|^{\tau-1}(1+j^{2s})<2(1+|j|^{2s+\tau-1})$ for all $|j|\geq1$. Then using \eqref{E5.26} and \eqref{E5.13}--\eqref{E6.27} yields that
\begin{align*}
\||\mathfrak{L}_{{1,\mathrm{D}}}|^{-\frac{1}{2}}{h}\|^2_{{s}}
\leq&\textstyle\frac{1}{\gamma L^2_1}\sum_{
1\leq|j|\leq N}(\lambda^{\pm}_0(\epsilon,\omega,w)+\Theta)(\hat{h}^{\pm}_{j,0})^2|j|^{\tau-1}(1+j^{2s})\\
&+\textstyle\frac{1}{\gamma L^2_1}\sum_{
1\leq|j|\leq N,l\geq1}(\lambda^{+}_l(\epsilon,\omega,w)+\Theta)(\hat{h}^{+}_{{j,l}})^2|j|^{\tau-1}(1+j^{2s})\\
&+\textstyle\frac{1}{\gamma L^2_1}\sum_{
1\leq|j|\leq N,l\geq1}(\lambda^{-}_l(\epsilon,\omega,w)+\Theta)(\hat{h}^{-}_{{j,l}})^2|j|^{\tau-1}(1+j^{2s})\\
\leq &\textstyle\frac{2L^2_2}{\gamma L^2_1}\|h\|^2_{{s+\frac{\tau-1}{2}}}\leq \frac{2L^2_2N^{\tau-1}}{\gamma L^2_1}\|h\|^2_{{s}}.
\end{align*}
This ends the proof of the lemma.
\end{proof}
The next step is to estimate the upper bounds of $\left\|\mathcal{R}_ih\right\|_{s'},i=1,2$. For $\tau\in(1,2),\gamma\in(0,1)$, we further impose  ``Melnikov's'' non-resonance conditions
\begin{align}\label{E5.11}
 |\frac{\omega {l}}{c}-{j}|>\frac{\gamma}{j^\tau},\quad\forall 1\leq j\leq N,~\forall l\geq 0.
\end{align}

If conditions \eqref{E5.12} and \eqref{E5.11} are provided, then we can claim

${\boldsymbol{\mathrm{(F1)}}}$:\quad Let $\sigma$ be as seen in \eqref{E0.17}. For fixed $\tau\in(1,2),\gamma\in(0,1)$, if  $j\neq k$, then there exists some constant $\tilde{L}>0$ such that
\begin{align*}
(\omega_j\omega_k)^{\pm}\geq{\gamma^4\tilde{L}^2}|j-k|^{-2\sigma},\quad\text{with  } (\omega_j\omega_k)^{+}:=\omega^+_j\omega^+_k,(\omega_j\omega_k)^{-}:=\omega^-_j\omega^-_k.
\end{align*}
\begin{lemm}
Let \eqref{E5.12} and \eqref{E5.11} hold. For $\|w\|_{s+\sigma}\leq1$, there exists some positive constant  $L$ such that for all $s'\geq s>{1}/{2}$,
\begin{align}\label{E5.32}
\textstyle\|\mathcal{R}_1h\|_{s'}\leq\frac{\epsilon L}{2\gamma^2}\left(\|h\|_{s'}+\|w\|_{s'+\sigma}\|h\|_s\right), \quad h\in W_{N}.
\end{align}
In particular,
\begin{align}\label{E5.32-2}
\textstyle\|\mathcal{R}_1h\|_{s}\leq\frac{\epsilon L}{2\gamma^2}\|h\|_{s}, \quad h\in W_{N}.
\end{align}
\end{lemm}
\begin{proof}
From  formulae \eqref{E5.27}--\eqref{E5.28} and the definition of $|\mathfrak{L}_{{1,\mathrm{D}}}|^{-\frac{1}{2}}$, it follows that
\begin{align*}
R_1h=&-\epsilon|\mathfrak{L}_{1,{\mathrm{D}}}|^{-\frac{1}{2}}\Bigg(\sum\limits_{\stackrel{1\leq|j|,|k|\leq N}{j\neq k}}\frac{\hat{h}^{\pm}_{j,0}}{|\omega^2\lambda^{\pm}_{0}(\epsilon,\omega,w)-j^2|^{\frac12}}\frac{b_{k-j}}{a}
\psi^{\pm}_0(\epsilon,\omega,w)e^{\mathrm{i}k x}\\
&\quad \quad\quad\quad\quad\quad+\sum\limits_{\stackrel{1\leq|j|,|k|\leq N}{j\neq k,l\geq1}}\frac{\hat{h}^{+}_{j,l}}{|\omega^2\lambda^{+}_{l}(\epsilon,\omega,w)-j^2|^{\frac12}}\frac{b_{k-j}}{a}
\psi^{+}_l(\epsilon,\omega,w)e^{\mathrm{i}k x}\\
&\quad \quad\quad\quad\quad\quad+
\sum\limits_{\stackrel{1\leq|j|,|k|\leq N}{j\neq k,l\geq1}}\frac{\hat{h}^{-}_{j,l}}{|\omega^2\lambda^{-}_{l}(\epsilon,\omega,w)-j^2|^{\frac12}}\frac{b_{k-j}}{a}
\psi^{-}_l(\epsilon,\omega,w)e^{\mathrm{i}k x}\Bigg)\\
=&-\epsilon\sum\limits_{\stackrel{1\leq|j|,|k|\leq N}{j\neq k}}\frac{\hat{h}^{\pm}_{j,0}}{|\omega^2\lambda^{\pm}_{0}(\epsilon,\omega,w)-k^2|^{\frac12}|\omega^2\lambda^{\pm}_{0}(\epsilon,\omega,w)-j^2|^{\frac12}}\frac{b_{k-j}}{a}\psi^{\pm}_0(\epsilon,\omega,w)e^{\mathrm{i}kx}\\
&-\epsilon\sum\limits_{\stackrel{1\leq|j|,|k|\leq N}{j\neq k,l\geq1}}\frac{\hat{h}^{+}_{j,l}}{|\omega^2\lambda^{+}_{l}(\epsilon,\omega,w)-k^2|^{\frac12}|\omega^2\lambda^{+}_{l}(\epsilon,\omega,w)-j^2|^{\frac12}}\frac{b_{k-j}}{a}\psi^{+}_l(\epsilon,\omega,w)e^{\mathrm{i}kx}\\
&-\epsilon\sum\limits_{\stackrel{1\leq|j|,|k|\leq N}{j\neq k,l\geq1}}\frac{\hat{h}^{-}_{j,l}}{|\omega^2\lambda^{-}_{l}(\epsilon,\omega,w)-k^2|^{\frac12}|\omega^2\lambda^{-}_{l}(\epsilon,\omega,w)-j^2|^{\frac12}}\frac{b_{k-j}}{a}\psi^{-}_l(\epsilon,\omega,w)e^{\mathrm{i}kx},
\end{align*}
Then
\begin{align*}
(R_1h)_{k}=&-\epsilon\sum\limits_{\stackrel{1\leq|j|\leq N}{j\neq k}}\frac{\hat{h}^{\pm}_{j,0}}{|\omega^2\lambda^{\pm}_{0}(\epsilon,\omega,w)-k^2|^{\frac12}|\omega^2\lambda^{\pm}_{0}(\epsilon,\omega,w)-j^2|^{\frac12}}\frac{b_{k-j}}{a}\psi^{\pm}_0(\epsilon,\omega,w)\\
&-\epsilon\sum\limits_{\stackrel{1\leq|j|\leq N}{j\neq k,l\geq 1}}\frac{\hat{h}^{+}_{j,l}}{|\omega^2\lambda^{+}_{l}(\epsilon,\omega,w)-k^2|^{\frac12}|\omega^2\lambda^{+}_{l}(\epsilon,\omega,w)-j^2|^{\frac12}}\frac{b_{k-j}}{a}\psi^{+}_l(\epsilon,\omega,w)\\
&-\epsilon\sum\limits_{\stackrel{1\leq|j|\leq N}{j\neq k,l\geq 1}}\frac{\hat{h}^{-}_{j,l}}{|\omega^2\lambda^{-}_{l}(\epsilon,\omega,w)-k^2|^{\frac12}|\omega^2\lambda^{-}_{l}(\epsilon,\omega,w)-j^2|^{\frac12}}\frac{b_{k-j}}{a}\psi^{-}_l(\epsilon,\omega,w).
\end{align*}
Hence combining $\mathrm{(\bf F1)}$ with \eqref{E6.3}, \eqref{scalar}, \eqref{E6.27} yields that
\begin{align*}
\textstyle\|(R_1h)_k\|_{H^1}
\leq\frac{\epsilon L_2}{\gamma^2 \tilde{L}L_1}\sum_{1\leq|j|\leq N,j\neq k}\|\frac{b_{k-j}}{a}\|_{H^1}|k-j|^\sigma\|h_j\|_{H^1}.
\end{align*}
We further define
\begin{align*}
&\textstyle\Lambda(x):=\sum_{1\leq |j|,|k|\leq N}\|\frac{b_{k-j}}{a}\|_{H^1}|k-j|^\sigma\|h_j\|_{H^1}e^{\mathrm{i}kx},\\ &\textstyle\mathrm{{p}}(x):=\sum_{j\in\mathbb{Z}}\|\frac{b_{j}}{a}\|_{H^1}|j|^{\sigma} e^{\mathrm{i}jx},\quad
\mathrm{{q}}(x):=\sum_{1\leq|j|\leq N}\|h_j\|_{H^1}e^{\mathrm{i}jx}.
\end{align*}
It is straightforward that $\Lambda=\mathrm{P}_{N}(\mathrm{pq})$. Moreover, due to \eqref{E5.14}, we have that for all $s'\geq s>{1}/{2}$,
\begin{align*}
\|\mathrm{p}\|_{s'}\leq C'(1+\|w\|_{s'+\sigma}),\quad
\|\mathrm{q}\|_{s'}=\|h\|_{s'}.
\end{align*}
Hence, if $\|w\|_{s+\sigma}\leq1$, then using \eqref{E2.1} shows that
\begin{align*}
\textstyle\|R_1h\|_{s'} \leq\frac{\epsilon C'}{2\gamma^2}(\|w\|_{s'+\sigma}\|h\|_s+\|h\|_{s'}).
\end{align*}
As a consequence, formula \eqref{E5.32} follows from the above inequality together with \eqref{E5.30}. Moreover we can easily derive  \eqref{E5.32-2}.
\end{proof}

\begin{lemm}
For $\|w\|_{s+\sigma}\leq1$, if conditions \eqref{E5.12} are given, then for all $s'\geq s>{1}/{2}$,
\begin{align}\label{E5.31}
\textstyle\|\mathcal{R}_2h\|_{s'}\leq\frac{\epsilon L}{2\gamma}\left(\|h\|_{s'}+\|w\|_{s'+\sigma}\|h\|_s\right),\quad h\in W_{N}.
\end{align}
In particular,
\begin{align}\label{E5.31-2}
\textstyle\|\mathcal{R}_2h\|_{s}\leq\frac{\epsilon L}{2\gamma}\|h\|_{s},\quad h\in W_{N}.
\end{align}
\end{lemm}
\begin{proof}
In view of Lemma \ref{lemma3.1}, one has
\begin{align*}
\mathrm{D}_{w}v(\epsilon,\omega, w)[|\mathfrak{L}_{1,{\mathrm{D}}}|^{-\frac{1}{2}}h]\in {H}^1(\mathbb{T}).
\end{align*}
Moreover, it can be obtained from the fact  $1<\tau<2$ that $\sigma=\frac{\tau(\tau-1)}{(2-\tau)}>\tau-1$. Then  combining \eqref{E5.28}, \eqref{E5.17} with \eqref{E5.15}--\eqref{E5.14} gives that
\begin{align*}
\|R_2h\|_{s'}\leq&\textstyle\frac{\epsilon\sqrt{2}L_2}{\sqrt{\gamma}L_1}\|\frac{1} {a}\|_{H^1}\|b\mathrm{D}_wv(\epsilon,\omega,w)[|\mathfrak{L}_{1,{\mathrm{D}}}|^{-\frac{1}{2}}h]\|_{s'+\frac{\tau-1}{2}}\\
\stackrel{\eqref{E2.1}}
{\leq}&\textstyle\frac{\epsilon \sqrt{2} L_2}{\sqrt{\gamma}L_1}\|\frac {1 } {a}\|_{H^1}C(\|b\|_{s'+\sigma}\||\mathfrak{L}_{1,{\mathrm{D}}}|^{-\frac{1}{2}}h\|_{s-\frac{\tau-1}{2}}
+\|b\|_{s+\sigma}\||\mathfrak{L}_{1,{\mathrm{D}}}|^{-\frac{1}{2}}h\|_{s-\frac{\tau-1}{2}})\\
{\leq}&\textstyle\frac{\epsilon C'}{2{\gamma}}(\|w\|_{s'+\sigma}\|h\|_s+\|h\|_{s'}).
\end{align*}
Hence we can get \eqref{E5.31} from the above estimate together with \eqref{E5.30}. Furthermore, it is clear that \eqref{E5.31-2} follows.
\end{proof}

\begin{lemm}\label{lemma3}
Let \eqref{E5.12} and \eqref{E5.11} hold.  If $\|w\|_{s+\sigma}\leq1$, for $\epsilon L\gamma^{-2}< \delta_0$ small enough, then the operator $\mathrm{Id}-\mathcal{R}$ is invertible and its inverse operator satisfies that for all $s'\geq s>{1}/{2}$,
\begin{align}\label{E5.21}
\|(\mathrm{Id}-\mathcal{R})^{-1}h\|_{s'}\leq 2(\|h\|_{s'}+\|w\|_{s'+\sigma}\|h\|_s), \quad h\in W_{N}.
\end{align}
\end{lemm}
\begin{proof}
Using \eqref{E5.32-2} and \eqref{E5.31-2} yields that for $\epsilon L\gamma^{-2}< \delta_0$ small enough,
\begin{align}\label{E5.20}
\textstyle{\|\mathcal{R}h\|_{s}\leq\epsilon L\gamma^{-2}\|h\|_{s}\leq\frac{1}{2}\|h\|_s}.
\end{align}
Then the operator $(\mathrm{Id}-\mathcal{R})$ is invertible from the Neumann series.

Next let us claim that if $\|w\|_{s+\sigma}\leq1$, then for all $\mathrm{p}\in\mathbb{N}^{+}$,
\begin{align}\label{E5.19}
\|\mathcal{R}^\mathrm{p}h\|_{s'}\leq(\epsilon L\gamma^{-2})^\mathrm{p}(\|h\|_{s'}+\mathrm{p}\|w\|_{s'+\sigma}\|h\|_s),\quad h\in W_{N}.
\end{align}
Hence, for $\epsilon L\gamma^{-2}< \delta_0$ small enough, the above inequality  implies that for all $s'\geq s>{1}/{2}$,
\begin{align*}
\|(\mathrm{Id}-\mathcal{R})^{-1}h\|_{s'}=&\textstyle\|(\mathrm{Id}+\sum_{\mathrm{p}\in\mathbb{N}^{+}}\mathcal{R}^\mathrm{p})h\|_{s'}\leq\|h\|_{s'}
+\sum_{\mathrm{p}\in\mathbb{N}^{+}}\|\mathcal{R}^\mathrm{p}h\|_{s'}\\
\leq&\textstyle\|h\|_{s'}+\sum_{\mathrm{p}\in\mathbb{N}^{+}}(\epsilon L\gamma^{-2})^\mathrm{p}(\|h\|_{s'}+\mathrm{p}\|w\|_{s'+\sigma}\|h\|_s)\\
\leq&\textstyle2\|h\|_{s'}+2\|w\|_{s'+\sigma}\|h\|_s.
\end{align*}

It remains to check \eqref{E5.19} by induction. For $\mathrm{p}=1$, formulae \eqref{E5.32} and \eqref{E5.31}  establish that
\begin{align}\label{E5.33}
\left\|\mathcal{R}h\right\|_{s'}\leq\epsilon L\gamma^{-2}\left(\|h\|_{s'}+\|w\|_{s'+\sigma}\|h\|_s\right).
\end{align}
This shows that \eqref{E5.19} holds for $\mathrm{p}=1$.
Suppose that \eqref{E5.19} could hold for $\mathrm{p}=\ell$, with $\ell\in\mathbb{N}$ and $\ell\geq2$. Let us show that \eqref{E5.19} follows for $\mathrm{p}=\ell+1$. Based on the assumption for $\mathrm{p}=\ell$, \eqref{E5.20} and \eqref{E5.33}, we can conclude
\begin{align*}
\|\mathcal{R}^{\ell+1}h\|_{s'}=&\|\mathcal{R}^\ell(\mathcal{R}h)\|_{s'}{\leq}
(\epsilon L\gamma^{-2})^\ell(\|\mathcal{R}h\|_{s'}+\ell\|w\|_{s'+\sigma}\|\mathcal{R}h\|_s)\\
\leq&(\epsilon L \gamma^{-2} )^l\left(\epsilon L \gamma^{-2}\|h\|_{s'}+(\ell\epsilon L\gamma^{-2} +\epsilon L\gamma^{-2} )\|w\|_{s'+\sigma}\|h\|_s\right)\\
\leq&(\epsilon L\gamma^{-2} )^{\ell+1}(\|h\|_{s'}+(\ell+1)\|w\|_{s'+\sigma}\|h\|_s).
\end{align*}
Thus we arrive at the conclusion of the lemma.
\end{proof}

As a consequence, Lemma \ref{lem:invertibility} follows from \eqref{E5.30}--\eqref{L:equivalent},  \eqref{E5.17.2} and \eqref{E5.21}. Namely we complete the proof of the invertibility  of the ``truncated-restricted'' of the linearized operator.

Finally, we need to check the fact $\mathrm{(\bf F1)}$. For this, we need to introduce more properties on the eigenvalues of the Hill problem \eqref{E6.1}. By Taylor expansion and \eqref{E6.24}, there is an integer $J_0>0$ large enough, $\Theta>0$ large enough such that
\begin{align}\label{E6.28}
\textstyle|(\lambda^{\pm}_l(\epsilon,\omega,w))^{\frac12}-\frac{l}{c}|\leq {\Theta}/{l},\quad \forall l>J_0.
\end{align}
Let $l^*$ be define by \eqref{E6.27}. If $j^2-\omega^2\lambda^{+}_{J_0+1}(\epsilon,\omega,w)>0$, then $l^*\geq J_0+1$.  Therefore there exists  an integer ${J}_1:={J}_1({J}_0)>0$ such that for every $j>\omega{J}_1$ (satisfying $l^*\geq J_0+1$),
\begin{align}\label{E6.25}
\textstyle l^*\geq {\Theta_0j}/{\omega},\quad \Theta_0>0.
\end{align}

Now let us complete the proof of the fact $\mathrm{(\bf F1)}$.

\begin{proof}[\underline{Proof of $\mathrm{(\bf F1)}$}]
The fact $\varsigma=\frac{2-\tau}{\tau}\in(0,1)$ follows from $1<\tau<2$. Denote
\begin{align*}
\omega^{\pm}_j:=|\omega^2\lambda^{\pm}_{l^{*}}(\epsilon,\omega,w)-j^2|,\quad\omega^{\pm}_k:=|\omega^2\lambda^{\pm}_{i^{*}}(\epsilon,\omega,w)-k^2|,\quad l^{*}\geq0,i^{*}\geq0,
\end{align*}
where $j\geq1,k\geq 1$, with $j\neq k$. Let us consider the following two cases.

{\bf Case 1}: $2|k-j|>(\max{\{k,j\}})^{\varsigma}$. Due to \eqref{E5.13}, we obtain
\begin{align*}
(\omega_j\omega_k)^{\pm}\geq\frac{\gamma^2}{(jk)^{\tau-1}}\geq\frac{\gamma^2}{(\max{\{k,j\}})^{2(\tau-1)}}
\geq\frac{\gamma^2}{2^{2(\tau-1)/\varsigma}|k-j|^{2(\tau-1)/\varsigma}}.
\end{align*}

{\bf Case 2}: $0<2|k-j|\leq(\max{\{k,j\}})^{\varsigma}$. Since $\varsigma\in(0,1)$, if $k>j$, then
\begin{align*}
2(k-j)\leq k^{\varsigma}\Longrightarrow 2j\geq2k-k^{\varsigma}> k.
\end{align*}
In the same way, if $j>k$, then $2k> j$. Therefore, $\frac{k}{2}< j<2k$.

$\mathrm{(i)}$ If $\lambda^{\pm}_{l^{*}}(\epsilon,\omega,w)<0$ and $\lambda^{\pm}_{i^{*}}(\epsilon,\omega,w)<0$, then $\omega^{\pm}_j\geq j^2,\omega^{\pm}_k\geq k^2$. Hence we carry out
\begin{align*}
(\omega_j\omega_k)^{\pm}\geq 1\geq\gamma^4.
\end{align*}

$\mathrm{(ii)}$ If either $\lambda^{\pm}_{l^{*}}(\epsilon,\omega,w)<0$ or $\lambda^{\pm}_{i^{*}}(\epsilon,\omega,w)<0$, then in the first case
\begin{align*}
(\omega_j\omega_k)^{\pm}\stackrel{\eqref{E5.13}}{\geq}j^2\frac{\gamma}{k^{\tau-1}}\geq2^{1-\tau}\gamma
\end{align*}
and in the latter
\begin{align*}
(\omega_j\omega_k)^{\pm}\stackrel{\eqref{E5.13}}{\geq}\frac{\gamma}{j^{\tau-1}}k^2\geq2^{1-\tau}\gamma.
\end{align*}

$\mathrm{(iii)}$ Let us consider the case $\lambda^{\pm}_{l^{*}}(\epsilon,\omega,w)>0$ and $\lambda^{\pm}_{i^{*}}(\epsilon,\omega,w)>0$. Moreover, we define
\begin{align*}
\textstyle j_{\star}:=\max\left\{2\omega{J}_1,\left(\frac{27\Theta}{2\gamma\Theta_0}\right)^{\frac{1}{1-\varsigma\tau}}\right\}.
\end{align*}

On the one hand, we assume that $\max{\{j,k\}}={j}> j_\star$. The definition of $j_{\star}$ implies that $l^*\geq J_0+1,i^*\geq J_0+1$. Moreover, it follows from the facts $\varsigma=\frac{2-\tau}{\tau},\tau\in(1,2)$ that $\varsigma\tau<1$. Using \eqref{E5.11}, \eqref{E6.28}--\eqref{E6.25} and the facts $\tau\in(1,2),\omega\in(\frac12,\frac32)$ yields that
\begin{align*}
|(\omega(\lambda^{\pm}_{l^{*}}(\epsilon,\omega,w))^{\frac12}-j)
-(\omega(\lambda^{\pm}_{i^{*}}(\epsilon,\omega,w))^{\frac12}-k)|\geq&\frac{\gamma}{(j-k)^\tau}-\omega\frac{\Theta}{ l^*}-\omega\frac{\Theta}{ i^*}\\
\geq &\frac{2^\tau \gamma}{j^{\varsigma\tau}}-\frac{27\Theta}{4\Theta_0j}
\geq\frac{\gamma}{2j^{\varsigma\tau}}+\frac{\gamma}{j^{\varsigma\tau}}.
\end{align*}
This arrives at
\begin{align*}
|\omega(\lambda^{\pm}_{l^{*}}(\epsilon,\omega,w))^{\frac12}-j|+|\omega(\lambda^{\pm}_{i^{*}}(\epsilon,\omega,w))^{\frac12}-k|
\geq\frac{1}{2}\left(\frac{\gamma}{j^{\varsigma\tau}}+\frac{\gamma}{k^{\varsigma\tau}}\right).
\end{align*}
Then
\begin{align*}
\mathrm{either}\quad|\omega(\lambda^{\pm}_{l^{*}}(\epsilon,\omega,w))^{\frac12}-j|\geq\frac{\gamma}{2j^{\varsigma\tau}}\quad\mathrm{or}
\quad|\omega(\lambda^{\pm}_{i^{*}}(\epsilon,\omega,w))^{\frac12}-k|\geq\frac{\gamma}{2k^{\varsigma\tau}}
\end{align*}
holds. Without loss of generality, suppose that $|\omega(\lambda^{\pm}_{l^{*}}(\epsilon,\omega,w))^{\frac12}-j|\geq\frac{\gamma}{2j^{\varsigma\tau}}$. Hence we get
\begin{align*}
\omega^{\pm}_j=|\omega^2\lambda^{\pm}_{l^{*}}(\epsilon,\omega,w)-j^2|=|(\omega (\lambda^{\pm}_{l^{*}}(\epsilon,\omega,w))^{\frac12}+j)(\omega (\lambda^{\pm}_{l^{*}}(\epsilon,\omega,w))^{\frac12}-j)|\geq \frac{\gamma}{2}j^{1-\varsigma\tau}.
\end{align*}
Thus it follows that
\begin{align*}
(\omega_j\omega_k)^{\pm}\stackrel{\eqref{E5.13}}{\geq} \frac{\gamma}{2}j^{1-\varsigma\tau} \frac{\gamma}{k^{\tau-1}}
\stackrel{k<2j}{\geq}\frac{\gamma^2}{2^{\tau}}j^{2-\tau-\varsigma\tau}=\frac{\gamma^2}{2^\tau},
\end{align*}
where $\varsigma$ is taken as $(2-\tau)/\tau$ to ensure $2-\tau-\varsigma\tau=0$. The same conclusion is reached if $\max{\{j,k\}}=k>j_\star$.

And on the other hand, we assume that $\max{\{j,k\}}\leq j_\star$. If $j_\star=2\omega J_1$, then for all $\omega\in(\frac{1}{2},\frac{3}{2})$,
\begin{align*}
(\omega_j\omega_k)^{\pm}\stackrel{\eqref{E5.13}}{\geq}\frac{\gamma^2}{(jk)^{\tau-1}}\geq\frac{\gamma^2}{(j_\star)^{2(\tau-1)}}
=\frac{\gamma^2}{(2\omega J_1)^{2(\tau-1)}}\geq\frac{\gamma^2}{(3J_1)^{2(\tau-1)}}.
\end{align*}
If $ j_\star=\left(\frac{27\Theta}{2\gamma\Theta_0}\right)^{\frac{1}{1-\varsigma\tau}}$, then for all $\omega\in(\frac{1}{2},\frac{3}{2})$,
\begin{align*}
(\omega_j\omega_k)^{\pm}\stackrel{\eqref{E5.13}}{\geq}\frac{\gamma^2}{(jk)^{\tau-1}}\geq\frac{\gamma^2}{(j_\star)^{2(\tau-1)}}
=\frac{\gamma^4\Theta^2_0}{\left({27\Theta}/2\right)^2}.
\end{align*}

Since $\omega^{\pm}_{j}=\omega^{\pm}_{-j},\omega^{\pm}_{k}=\omega^{\pm}_{-k}$, the remainder of the lemma may be proved in the similar way as shown above when $j\geq1,k\leq-1$, or $j\leq-1,k\geq1$, or $j,k\leq-1$.

As a result, we arrive at  the conclusion in $\mathrm{(\bf F1)}$.
\end{proof}

\section{Solutions of the range equation}\label{sec:range}

The purpose of the present section is  to look for solutions of  the rang equation \eqref{E:rang22} in the space $W\cap H^s$.  We first introduce the smoothness of the composition operator $\mathcal{F}_1$ with $\mathcal{F}_1(\epsilon,\omega,w):=\mathcal{F}(\epsilon,\omega,w)+ g$. Let $\ell\geq s'+3$ with $s'\geq s\geq 1/2$. If $f \in\mathcal{C}_\ell$ and $x\longmapsto g(\cdot,x)$ is in $C^\ell(\mathbb{T};H^1(\mathbb{T}))$, then  the operator $\mathcal{F}_1$  has the following standard properties.

$(\boldsymbol {\mathrm{P1}})$(Regularity.) $\mathcal{F}_1\in C^2(W\cap H^s;H^s)$ and $\mathcal{F}_1,\mathrm{D}_{w}\mathcal{F}_1,\mathrm{D}^2_{w}\mathcal{F}_1$ are bounded on $\{\|w\|_{s}\leq 1\}$. Note that  $\mathrm{D}_{w}\mathcal{F}_1=\mathrm{D}_{w}\mathcal{F},\mathrm{D}^2_{w}\mathcal{F}_1=\mathrm{D}^2_{w}\mathcal{F}$.

$(\boldsymbol {\mathrm{P2}})$(Tame.) $\forall s\leq s'\leq k-3$, $\forall w\in W\cap H^{s'}$ with $\|w\|_{s}\leq1$, one has
\begin{align*}
&\|\mathcal{F}_1(\epsilon,\omega,w)\|_{{s'}}\leq C(s')(1+\|w\|_{{s'}}),~~\|\mathrm{D}_{w}\mathcal{F}_1(\epsilon,\omega,w)[h]\|_{{s'}}{\leq}C(s')(\|w\|_{{s'}}\|h\|_{s}+\|h\|_{{s'}}),\\
&\|\mathrm{D}^2_{w}\mathcal{F}_1(\epsilon,\omega,w)[h,\mathrm{h}]\|_{s'}{\leq} C(s')(\|w\|_{s'}\|h\|_{s}\|\mathrm{h}\|_{s}+\|h\|_{s'}\|\mathrm{h}\|_{s}+\|h\|_{s}\|\mathrm{h}\|_{s'}).
\end{align*}

$(\boldsymbol {\mathrm{P3}})$(Taylor tame.) $\forall s\leq s'\leq k-3$, $\forall w\in W\cap H^{s'}$, $\forall h\in W\cap H^{s'}$ and $\|w\|_{s}\leq1$, $\|h\|_{s}\leq1$, one has
\begin{align*}
&\|\mathcal{F}_1(\epsilon,\omega,w+h)-\mathcal{F}_1(\epsilon,\omega,w)-\mathrm{D}_{w}\mathcal{F}_1(\epsilon,\omega,w)[h]\|_{s}\leq C\|h\|^2_{s},\\
&\|\mathcal{F}_1(\epsilon,\omega,w+h)-\mathcal{F}_1(\epsilon,\omega,w)-\mathrm{D}_{w}\mathcal{F}_1(\epsilon,\omega,w)[h]\|_{{s'}}\leq C(s')(\|w\|_{{s'}}\|h\|^2_{s}+\|h\|_{s}\|h\|_{{s'}}).
\end{align*}
Note that $(\mathrm{\bf P1})$--$(\mathrm{\bf P3})$ follow from Lemmata \ref{lemma2.3}--\ref{lemma2.4} and Lemma \ref{lemma3.1}.

Denote by the symbol $\lfloor\cdot\rfloor$ the integer part. For fixed $\chi>1$, let
\begin{align}\label{E4.5}
N_n:=\lfloor e^{\mathfrak{q}\chi^{n}}\rfloor,\quad\text{with }\mathfrak{q}=\ln N_0.
\end{align}
Moreover we introduce an open set
\begin{align*}
\textstyle A_{0}:=\left\{(\epsilon,\omega)\in(\epsilon_1,\epsilon_2)\times(\frac{1}{2},\frac{3}{2}):|\omega (\lambda^{\pm}_{l})^{\frac12}-j|>\frac{\gamma}{j^{\tau}},\forall j=1,\cdots,N_0, \forall l\geq0\right\},
\end{align*}
where $\lambda^{\pm}_l,l\geq0$ are the eigenvalues of the Hill problem
\begin{align}\label{E3.3}
\begin{cases}
-y''=\lambda ay,\\
y(t)=y(t+2\pi),\quad y'(t)=y'(t+2\pi).
\end{cases}
\end{align}
Notice that we denote $\lambda_0$ by $\lambda^{\pm}_0$. The following lemma addresses an inductive scheme.

\begin{lemm}(Inductive scheme)\label{lemma4.1}
For all $n\in\mathbb{N}$, there exists a sequence of subsets $(\epsilon,\omega)\in A_{n}\subseteq A_{n-1}\subseteq\cdots\subseteq A_1\subseteq A_{0}$, where
\begin{align*}
A_{n}:=\left\{(\epsilon,\omega)\in A_{n-1}:(\epsilon,\omega)\in\Delta^{\gamma,\tau}_{N_{n}}(w_{n-1})\right\},
\end{align*}
and a sequence $w_{n}(\epsilon,\omega)\in W_{N_{n}}$ with
\begin{align}\label{E4.10}
\|w_{n}\|_{s+\sigma}\leq1,
\end{align}
and
\begin{align}
&\textstyle\|w_{0}\|_{s}\leq \frac{K_1\epsilon}{\gamma}N^{\tau-1}_0,\quad
\|w_k-w_{k-1}\|_{s}\leq \frac{ K_2\epsilon}{\gamma}N^{-\sigma-1}_{k}, \quad\forall1\leq k\leq n,\label{E4.11}\\
&\textstyle\|\partial_{tt}w_0\|_s\leq\frac{K'_1\epsilon}{\gamma\omega^2}N^{\tau+1}_0, \quad \|\partial_{tt}(w_k-w_{k-1})\|_{s}\leq\frac{K'_2\epsilon}{\gamma\omega^2}N^{-1}_{k}, \quad\forall1\leq k\leq n,\label{E3.70}
\end{align}
where $K_i,K'_i>0,i=1,2$ (depending on $a, f, \epsilon_0, \hat{v}, \gamma,\tau, s,\beta$ at most) such that for $(\epsilon,\omega)\in A_{n}$, if  $\epsilon\gamma^{-2}< \delta_4$ is small enough, with ${\epsilon \gamma^{-1}K_1 N^{\tau}_0}<r$ and $\epsilon\gamma^{-1}K_2<r$, then $w_{n}(\epsilon,\omega)$ is a solution of equation
\begin{align}\label{E0.19}
  L_{\omega}w-\epsilon \mathrm{P}_{N_n}\Pi_{W}\mathcal{F}_1(\epsilon,\omega,w)=0. \tag{$P_{N_{n}}$}
\end{align}
\end{lemm}

\begin{proof}
The proof of the lemma is given by induction.

$\underline{Initialization}.$ Clearly, it can be seen that  the spectrum of  $\frac{1}{a}L_{\omega}$ on $W_{N_0}$ is
\begin{align*}
-\omega^2\lambda^{\pm}_{l}+j^2,\quad \forall1\leq|j|\leq N_0,\forall l\geq0,
\end{align*}
where $\lambda^{\pm}_{l}$ are the eigenvalues of \eqref{E3.3}. Then the definition of $A_0$ shows that solving equation $(P_{N_0})$ is reduced to look for  the fixed point of $w=U_0(w)$, where
\begin{align*}
\textstyle U_0:W_{N_0}\longrightarrow W_{N_0},\quad w\longmapsto \epsilon (\frac{1}{a}L_{\omega})^{-1}\frac{1}{a} \mathrm{P}_{N_0}\Pi_{W}\mathcal{F}_1(\epsilon,\omega,w).
\end{align*}
Let us check that $U_0$ is a contraction mapping.
\begin{lemm}\label{lemma1}
Let $(\epsilon,\omega)\in A_{0}$. If $\epsilon\gamma^{-2}< \delta_1N^{1-\tau}_0\leq\frac{\delta_0}{L}$ is small enough, then the mapping $U_0$ is a contraction in
\begin{align*}
\mathcal{B}(0,\rho_0):=\left\{w\in W_{N_0}:\|w\|_{s}\leq \rho_0\right\},\quad\rho_0:={\epsilon \gamma^{-1}K_1 N^{\tau-1}_0},
\end{align*}
with ${\epsilon \gamma^{-1}K_1 N^{\tau}_0}<r$.
\end{lemm}
Remark that the fact $\delta_1N^{1-\tau}_0\leq\frac{\delta_0}{L}$ is to satisfy the assumption in Lemma \ref{lemma3}.
\begin{proof}
By the definition of $A_0$ and $\mathrm{(\bf P1)}$, for $\epsilon\gamma^{-1}N^{\tau-1}_0\leq\epsilon\gamma^{-2}N^{\tau-1}_0<\delta_1$ small enough, one has
\begin{align}
&\textstyle\|U_0(w)\|_{s}{\leq}{\epsilon\gamma^{-1}\tilde{K} N^{\tau-1}_0}\|\frac{1}{a}\|_{H^1}\left\|\mathrm{P}_{N_0}\Pi_{W}\mathcal{F}_1(\epsilon,\omega,w)\right\|_s
\leq{\epsilon \gamma^{-1}K_1 N^{\tau-1}_0},\nonumber\\
&\textstyle\|\mathrm{D}U_0(w)\|_s{\leq}{\epsilon\gamma^{-1}\tilde{K} N^{\tau-1}_0}\|\frac{1}{a}\|_{H^1}\left\|\mathrm{P}_{N_0}\Pi_{W}\mathrm{D}_w\mathcal{F}_1(\epsilon,\omega,w)\right\|_s\leq{1}/{2}.\label{E3.24}
\end{align}
Thus the mapping $U_0$ is a contraction in $\mathcal{B}(0,\rho_0)$.
\end{proof}

Let us continue to the proof of initialization. Denote by $w_0$ the unique solution of equation $(P_{N_0})$ in $\mathcal{B}(0,\rho_0)$. Observe that
\begin{align*}
\omega^2\partial_{tt}w_0=a\partial_{xx}w_0+\epsilon \mathrm{P}_{N_0}\Pi_{W}\mathcal{F}_1(\epsilon,\omega,w_0).
\end{align*}
By virtue of \eqref{E4.52}, Lemma \ref{lemma1} and $\gamma\in (0,1)$, we conclude $\|\partial_{tt}w_0\|_s\leq\frac{K'_1\epsilon}{\gamma\omega^2}N^{\tau+1}_0.$
Moreover,  one has that for $\epsilon\gamma^{-1}N^{\tau-1}_0<\delta_1$  small enough,
\begin{align}\label{E3.12}
\textstyle\|w_0\|_{s+\beta}=\|\epsilon(\frac{1}{a}L_{\omega})^{-1}\frac{1}{a}\mathrm{P}_{N_0}\Pi_{W}\mathcal{F}_1(\epsilon,\omega,w_0)\|_{s+\beta}\leq \bar{K}.
\end{align}

$\underline{Iteration}.$ Suppose that we could obtain a solution $w_{n}\in W_{N_{n}}$ of \eqref{E0.19} satisfying estimates \eqref{E4.10}--\eqref{E3.70}.
Our purpose is to find a solution $w_{n+1}\in W_{N_{n+1}}$ of equation
\begin{align}\label{E0.18}
L_{\omega}w-\epsilon \mathrm{P}_{N_{n+1}}\Pi_{W}\mathcal{F}_1(\epsilon,\omega,w)=0 \tag{$P_{N_{n+1}}$}
\end{align}
with estimates \eqref{E4.10}--\eqref{E3.70} at ($n+1$)--th step. Let
\begin{align*}
w_{n+1}=w_{n}+h\quad\text{ with }h\in W_{N_{n+1}}
\end{align*}
be a solution of equation \eqref{E0.18}. It follows from the equality
$L_{\omega}w_{n}-\epsilon \mathrm{P}_{N_{n}}\Pi_{W}\mathcal{F}_1(\epsilon,\omega,w_{n})=0$ that
\begin{align*}
L_{\omega}(w_{n}+h)-\epsilon \mathrm{P}_{N_{n+1}}\Pi_{W}\mathcal{F}_1(\epsilon,\omega,w_{n}+h)=&L_{\omega}h+ L_{\omega}w_{n}-\epsilon \mathrm{P}_{N_{n+1}}\Pi_{W}\mathcal{F}_1(\epsilon,\omega,w_{n}+h)\\
=&-\mathcal{L}_{N_{n+1}}(\epsilon,\omega,w_n)h+R_{n}(h)+r_{n},
\end{align*}
where $\mathcal{L}_{N_{n+1}}(\epsilon,\omega,w_n)$ is defined by \eqref{E4.32} and
\begin{align*}
&R_{n}(h):=-\epsilon \mathrm{P}_{N_{n+1}}(\Pi_{W}\mathcal{F}_1(\epsilon,\omega,w_n+h)-\Pi_{W}\mathcal{F}_1(\epsilon,\omega,w_n)-\Pi_{W}\mathrm{D}_{w}\mathcal{F}_1(\epsilon,\omega,w_n)[h]),\\
&r_{n}:=\epsilon(\mathrm{P}_{N_{n}}\Pi_{W}\mathcal{F}_1(\epsilon,\omega,w_n)-\mathrm{P}_{N_{n+1}}\Pi_{W}\mathcal{F}_1(\epsilon,\omega,w_n))=
-\epsilon \mathrm{P}^{\bot}_{N_{n}}\mathrm{P}_{N_{n+1}}\Pi_{W}\mathcal{F}_1(\epsilon,\omega,w_n).
\end{align*}
Since $(\epsilon,\omega)\in A_{n+1}\subseteq A_{n}$, by using \eqref{E4.10} and Lemma \ref{lem:invertibility}, the  operator $\mathcal{L}_{N_{n+1}}(\epsilon,\omega,w_n)$ is invertible with, for all $s'\geq s>1/2$,
\begin{align}
&\textstyle\|\mathcal{L}^{-1}_{N_{n+1}}(\epsilon,\omega,w_{n})h\|_{s}\leq K\gamma^{-1}N_{n+1}^{\tau-1}\|h\|_{s},\label{E4.13}\\
&\textstyle\|\mathcal{L}^{-1}_{N_{n+1}}(\epsilon,\omega,w_{n})h\|_{s'}\leq K(s')\gamma^{-1}N_{n+1}^{\tau-1}(\|h\|_{s'}+\|w\|_{s'+\sigma}\|h\|_{s}).\label{E4.12}
\end{align}
Then we reformulate solving equation \eqref{E0.18} as  the fixed point problem of  $h=U_{n+1}(h)$, where
\begin{align*}
U_{n+1}: W_{N_{n+1}} \longrightarrow W_{N_{n+1}},\quad h\longmapsto \mathcal{L}^{-1}_{N_{n+1}}(\epsilon,\omega,w_n)(R_{n}(h)+r_{n}).
\end{align*}

Moreover we denote $S_{n}:=1+\|w_n\|_{s+\beta}.$ If $\epsilon\gamma^{-1}\leq\epsilon\gamma^{-2}<\delta_2$ is small enough, then we can claim that for some positive constant $\tilde{C}:=\tilde{C}(\chi,\mathfrak{q},\tau,\sigma)$,
\begin{align*}
\textstyle \boldsymbol{\mathrm{(F2)}}:\quad S_{n}\leq \tilde{C}(1+\bar{K})N^{\frac{1}{\chi-1}{(\tau-1+\sigma)}}_{n+1}.
\end{align*}
The proof of  $\mathrm{(\bf F2)}$ will be given in Lemma \ref{lemma4.3}.

Similar to Lemma \ref{lemma1}, we now prove that the mapping $U_{n+1}$ is a contraction.

\begin{lemm}\label{lemma2}
For $(\epsilon,\omega)\in A_{n+1}$ and $\epsilon\gamma^{-2}<\delta_2\leq\delta_1N^{1-\tau}_0$ small enough, there exists $K_2>0$  such that the mapping $U_{n+1}$ is a contraction in
\begin{align}\label{E4.33}
\mathcal{B}(0,\rho_{n+1}):=\left\{h\in W_{N_{n+1}}:\|h\|_{s}\leq\rho_{n+1}\right\},\quad \rho_{n+1}:={\epsilon \gamma^{-1} K_2}N_{n+1}^{-\sigma-1},
\end{align}
with $\epsilon\gamma^{-1}K_2<r$. Moreover, the unique fixed point ${h}_{n+1}(\epsilon,\omega)$ of $U_{n+1}$ also satisfies
\begin{align}\label{E3.11}
\|{h}_{n+1}\|_{s}\leq{\epsilon\gamma^{-1} K_2}N^{\tau-1}_{n+1}N^{-\beta}_{n}S_{n}.
\end{align}
\end{lemm}

\begin{proof}
It follows from \eqref{E4.53} and $\mathrm{(\bf P2)}$--$\mathrm{(\bf P3)}$ that
\begin{align*}
\|r_n\|_s{\leq}\epsilon C(\beta)N^{-\beta}_{n}S_n,\quad\|R_{n}(h)\|_s\leq\epsilon C\|h\|^2_s.
\end{align*}
Then combining \eqref{E4.13} with \eqref{E4.33} yields that for $\epsilon\gamma^{-1}<\delta_2$ small enough,
\begin{align}
\|U_{n+1}(h)\|_{s}
\leq&\textstyle\frac{\epsilon {K}'}{\gamma}N^{\tau-1}_{n+1}\|h\|^2_s+\frac{\epsilon {K}'}{\gamma}N^{\tau-1}_{n+1}N^{-\beta}_{n}S_{n}\label{E3.22}\\
\leq&\textstyle\frac{\epsilon {K}'}{\gamma}N^{\tau-1}_{n+1}\rho^2_{n+1}+\frac{\epsilon {K}'}{\gamma}N^{\tau-1}_{n+1}N^{-\beta}_{n}S_{n}.\nonumber
\end{align}
From the definition of $\rho_{n+1}$, $\mathrm{(\bf F2)}$ and \eqref{E5.35}, we derive
\begin{align}\label{E3.23}
\textstyle\frac{\epsilon {K}'}{\gamma}N^{\tau-1}_{n+1}\rho_{n+1}\leq\frac{1}{2},\quad\frac{\epsilon {K}'}{\gamma}N^{\tau-1}_{n+1}N^{-\beta}_{n}S_{n}\leq\frac{\rho_{n+1}}{2}.
\end{align}
This shows that $\|U_{n+1}(h)\|_{s}\leq \rho_{n+1}$. Moreover, taking the derivative  of  $U_{n+1}$ with respect to $h$ yields that
\begin{align}\label{E4.15}
\mathrm{D}U_{n+1}(h)[\mathfrak{w}]=-\epsilon\mathcal{L}^{-1}_{N_{n+1}}(\epsilon,\omega,w_n)\mathrm{P}_{N_{n+1}}(\Pi_{W}\mathrm{D}_{w}
\mathcal{F}_1(\epsilon,\omega,w_n+h)[\mathfrak{w}]-\Pi_W\mathrm{D}_{w}\mathcal{F}_1(\epsilon,\omega,w_n)[\mathfrak{w}]).
\end{align}
Consequently, for $\epsilon\gamma^{-1}<\delta_2$ small enough, it follows from $\mathrm{(\bf P1)}\text{--}\mathrm{(\bf P2)}$ and \eqref{E4.13} that
\begin{align}\label{E4.16}
\textstyle\|\mathrm{D}U_{n+1}(h)[\mathfrak{w}]\|_{s}\leq\frac{{K}'\epsilon}{\gamma}N^{\tau-1}_{n+1}
\|h\|_{s}\|\mathfrak{w}\|_{s}\leq\frac{{K}'\epsilon}{\gamma}N^{\tau-1}_{n+1}\rho_{n+1}\|\mathfrak{w}\|_{s}
\leq\frac{1}{2}\|\mathfrak{w}\|_{s}.
\end{align}
Hence $U_{n+1}$ is a contraction in $\mathcal{B}(0,\rho_{n+1})$.

 Let ${h}_{n+1}(\epsilon,\omega)$ denote the unique fixed point of $U_{n+1}$. By virtue of \eqref{E4.33}, \eqref{E3.22}--\eqref{E3.23}, we obtain
\begin{align*}
\textstyle\|h_{n+1}\|_{s}\leq\frac{1}{2}\|h_{n+1}\|_s+\frac{\epsilon {K}'}{\gamma}N^{\tau-1}_{n+1}N^{-\beta}_{n}S_{n}.
\end{align*}
This implies formula \eqref{E3.11}. Hence the proof is accomplished.
\end{proof}

Finally, let us complete the proof of Lemma \ref{lemma4.1}.
Since $w_n,w_{n+1}$, with $w_{n+1}=w_n+h_{n+1}$, are solutions of  equations \eqref{E0.19}, \eqref{E0.18}, respectively, it follows that
\begin{align}\label{E4.56}
\omega^2\partial_{tt}h_{n+1}=a\partial_{xx}h_{n+1}+\epsilon \mathrm{P}_{N_{n+1}}(\Pi_{W}\mathcal{F}_1(\epsilon,\omega,w_n+h_{n+1})-\Pi_{W}\mathcal{F}_1(\epsilon,\omega,w_n))&\nonumber\\
+\epsilon \mathrm{P}^{\perp}_{N_{n}}\mathrm{P}_{N_{n+1}}\Pi_{W}\mathcal{F}_1(\epsilon,\omega,w_n)&.
\end{align}
From $\gamma\in(0,1)$ and \eqref{E4.52}--\eqref{E4.53}, we establish that
\begin{align}\label{E4.57}
\textstyle\|\partial_{tt}h_{n+1}\|_{s}\stackrel{\mathrm{(\bf P2)}}{\leq}\frac{C'(\beta)}{\omega^{2}}
(\|h_{n+1}\|_{s+2}+\epsilon\|h_{n+1}\|_s+\epsilon N^{-\beta}_nS_n)
\stackrel{\eqref{E3.11}}{\leq}\frac{C''(\beta)\epsilon}{\gamma\omega^2}N^{\tau+1}_{n+1}N^{-\beta}_{n}S_{n}.
\end{align}
Then, by means of  \eqref{E5.35} and $\mathrm{(\bf F2)}$, we get
\begin{align*}
\textstyle\|\partial_{tt}h_{n+1}\|_{s}{\leq}\frac{K'_2\epsilon}{\gamma\omega^2}N^{-1}_{n+1}.
\end{align*}
Moreover, if $\epsilon\gamma^{-2}<\delta_3\leq\delta_2$ is small enough, then Lemmata \ref{lemma1}--\ref{lemma2} show that
\begin{align*}
\textstyle \|w_{n+1}\|_{s+\sigma}\leq\sum_{k=0}^{n+1}\|h_{k}\|_{s+\sigma}\stackrel{\eqref{E4.52}}
{\leq}\sum_{k=1}^{n+1}N^{\sigma}_{k}\frac{K_2\epsilon N^{-\sigma-1}_{k}}{\gamma}+N^{\sigma}_0\frac{\epsilon K_1 N^{\tau-1}_0}{\gamma }\leq1.
\end{align*}
This indicates  that \eqref{E4.10} holds  at ($n+1$)--th step.

As a consequence, we complete the proof of the inductive lemma.
\end{proof}

In addition, we also need to  obtain the upper bounds on the derivatives of $h_k$ with respect to $\omega$ on $s$-norm.
\begin{lemm}
For $(\epsilon,\omega)\in A_{k}$, if $\epsilon{\gamma^{-2}}<\delta_4\leq\delta_3$ is small enough, then the mapping $h_{k}=w_k-w_{k-1}$ is in $C^1(A_{k}\cap\{\epsilon<\delta_4\gamma^2\};W_{N_{k}})$ with
\begin{align}\label{E3.30}
\quad\|\partial_{\omega}w_0(\epsilon,\cdot)\|_{s}\leq \epsilon\gamma^{-2}K_3,
\quad\|\partial_{\omega}h_k(\epsilon,\cdot)\|_{s}\leq \epsilon\gamma^{-2}K_4 N_{k}^{-1},\quad\forall1\leq k\leq n
\end{align}
for some constants $K_3,K_4>0$.

\end{lemm}
\begin{proof}
The proof of the lemma is given by induction.
First, let us define
\begin{align*}
\textstyle\mathcal{U}_{0}(\epsilon,\omega,w):=w-\epsilon (\frac{1}{a}L_{\omega})^{-1}\frac{1}{a}\mathrm{P}_{N_0}\Pi_{W}\mathcal{F}_1(\epsilon,\omega,w).
\end{align*}
It follows from Lemma \ref{lemma1} that $\mathcal{U}_{0}(\epsilon,\omega,w_0)=0$, which leads to
\begin{align*}
\textstyle \mathrm{D}_{w}\mathcal{U}_{0}(\epsilon,\omega,w_0)=&\textstyle\mathrm{Id}-\epsilon (\frac{1}{a}L_{\omega})^{-1}\frac{1}{a}\mathrm{P}_{N_0}\Pi_{W}\mathrm{D}_{w}\mathcal{F}_1(\epsilon,\omega,w_0)\\
=&\textstyle\mathrm{Id}-\epsilon (\frac{1}{a}L_{\omega})^{-1}\frac{1}{a}\mathrm{P}_{N_0}\Pi_{W}\mathrm{D}_{w}\mathcal{F}(\epsilon,\omega,w_0).
\end{align*}
For $\epsilon{\gamma^{-2}}<\delta_4$ ($\epsilon\gamma^{-1}N^{\tau-1}_0< \delta_1$) small enough, it is invertible due to \eqref{E3.24}. Obviously, we can get  $w_0\in C^1(A_0\cap\{\epsilon<\delta_4\gamma^2\};W_{N_0})$ by the implicit function theorem. Taking the derivative of  $\mathcal{U}_{0}(\epsilon,\omega,w_0)=0$  with respect to $\omega$ yields that
\begin{align}
\partial_{\omega}w_{0}=&\textstyle\epsilon\partial_{\omega}(\frac{1}{a}L_{\omega})^{-1}\frac{1}{a}\mathrm{P}_{N_0}\Pi_{W}f(v(\epsilon,\omega,w_0)+w_0)\nonumber\\
&+
\textstyle\epsilon(\frac{1}{a}L_{\omega})^{-1}\frac{1}{a}\mathrm{P}_{N_0}\Pi_{W}f'(v(\epsilon,\omega,w_0)+w_0)(\partial_\omega v(\epsilon,\omega,w_0)+\mathrm{D}_{w}v(\epsilon,\omega,w_0)[\partial_{\omega}w_0]+\partial_{\omega}w_0).\nonumber
\end{align}
Moreover, taking the derivative of the identity $(\frac{1}{a}L_{\omega})(\frac{1}{a}L_{\omega})^{-1}\mathfrak{w}=\mathfrak{w}$ with respect to $\omega$ reads that
\begin{align*}
\textstyle \partial_{\omega}(\frac{1}{a}L_{\omega})^{-1}\mathfrak{w}=-(\frac{1}{a}L_{\omega})^{-1}
(\frac{2\omega}{a}\partial_{tt})(\frac{1}{a}L_{\omega})^{-1}\mathfrak{w}.
\end{align*}
From  the definition of $A_0$, Lemma \ref{lemma3.1} and Lemma \ref{lemma2.4}, we have that for $\epsilon{\gamma^{-2}}<\delta_4$ small enough,
\begin{align*}
\textstyle\|\partial_{\omega}w_0(\epsilon,\cdot)\|_{s}\stackrel{\eqref{E3.70},\eqref{E2.5}}{\leq}{K'\epsilon}{\gamma^{-2}}+\frac{1}{2}\|w_0(\epsilon,\cdot)\|_s,
\end{align*}
which carries out $\|\partial_{\omega}w_0(\epsilon,\cdot)\|_s\leq{\epsilon}{\gamma^{-2}K_3}$. Combining this with \eqref{E4.52} presents that for $\epsilon\gamma^{-2}< \delta_4$ small enough,
\begin{align}\label{E3.28}
\|\partial_{\omega}w_0(\epsilon,\cdot)\|_{s+\beta}{\leq}{\bar{K}}{\gamma^{-1}}.
\end{align}

Next, suppose that we could obtain a solution $h_{k}\in C^1(A_{k}\cap\{\epsilon<\delta_4\gamma^2\};W_{N_{k}}),1\leq k\leq n$ satisfying \eqref{E3.30}. Hence it is straightforward that
\begin{align}\label{E4.39}
\|\partial_{\omega}w_{n}(\epsilon,\cdot)\|_{s}\leq{\bar{K}\epsilon}{\gamma^{-2}}
\end{align}
for some $\bar{K}>0$. Let us define
\begin{align*}
\mathcal{U}_{n+1}(\epsilon,\omega,h):=-L_{\omega}(w_{n}+h)+\epsilon \mathrm{P}_{N_{n+1}}\Pi_{W}\mathcal{F}_1(\epsilon,\omega,w_n+h).
\end{align*}
Lemma \ref{lemma2} shows that $\mathcal{U}_{n+1}(\epsilon,\omega,h_{n+1})=0$. Combining this with {\eqref{E4.32}} and {\eqref{E4.15}} gives rise to
\begin{align}\label{E8.5}
\mathrm{D}_{h}\mathcal{U}_{n+1}(\epsilon,\omega,h_{n+1})
=\mathcal{L}_{N_{n+1}}(\epsilon,\omega,w_{n+1})=\mathcal{L}_{N_{n+1}}(\epsilon,\omega,w_{n})(\mathrm{Id}-\mathrm{D}U_{n+1}(h_{n+1})).
\end{align}
Then, because of \eqref{E4.16}, $\mathcal{L}_{N_{n+1}}(\epsilon,\omega,w_{n+1})$ is invertible with
\begin{align}\label{E4.19}
\|\mathcal{L}^{-1}_{N_{n+1}}(\epsilon,\omega,w_{n+1})\mathfrak{w}\|_{s}
\stackrel{\eqref{E4.13}}{\leq}{2K}{\gamma^{-1}}N_{n+1}^{\tau-1}\|\mathfrak{w}\|_{s}.
\end{align}
Hence, from the  implicit function theorem, it follows that
 \[h_{n+1}\in C^1(A_{n+1}\cap\{\epsilon<\delta_4\gamma^2\};W_{N_{n+1}}).\]
As a result,
\begin{align*}
\partial_{\omega}\mathcal{U}_{n+1}(\epsilon,\omega,h_{n+1})+\mathrm{D}_{h}\mathcal{U}_{n+1}(\epsilon,\omega,h_{n+1})\partial_{\omega}h_{n+1}=0.
\end{align*}
This gives rise to
\begin{align}\label{E4.20}
\partial_{\omega}h_{n+1}=-\mathcal{L}^{-1}_{N_{n+1}}(\epsilon,\omega,w_{n+1})\partial_{\omega}\mathcal{U}_{n+1}(\epsilon,\omega,h_{n+1}),
\end{align}
where
\begin{align}\label{E3.31}
\partial_{\omega}\mathcal{U}_{n+1}(\epsilon,\omega,h_{n+1})=&-2\omega(h_{n+1})_{tt}+\epsilon \mathrm{P}^{\bot}_{N_{n}}\mathrm{P}_{N_{n+1}}\Pi_{W}\partial_{\omega}\mathcal{F}_1(\epsilon,\omega,w_n)\nonumber\\
&+\epsilon \mathrm{P}_{N_{n+1}}\left(\Pi_{W}\partial_{\omega}\mathcal{F}_1(\epsilon,\omega,w_{n}+h_{n+1})-\Pi_{W}\partial_{\omega}\mathcal{F}_1(\epsilon,\omega,w_n)\right).
\end{align}
For $\epsilon\gamma^{-2}< \delta_4$ small enough, because of Lemma \ref{lemma3.1}, Lemma \ref{lemma2.4} and formula \eqref{E2.1}, one has
\begin{align*}
&\|\Pi_{W}\partial_{\omega}\mathcal{F}_1(\epsilon,\omega,w_n)\|_{s+\beta}{\leq} C(\beta)\|w_n\|_{s+\beta}(1+\|\partial_{\omega}w_n\|_s)+C(\beta)(1+\|\partial_{\omega}w_n\|_{s+\beta}),\\
&\|\Pi_{W}\partial_{\omega}\mathcal{F}_1(\epsilon,\omega,w_{n}+h_{n+1})-\Pi_{W}\partial_{\omega}\mathcal{F}_1(\epsilon,\omega,w_n)\|_{s}\leq C(1+\|\partial_{\omega} w_n\|_{s})\|h_{n+1}\|_s,\\
&\|\Pi_{W}\partial_{\omega}\mathcal{F}_1(\epsilon,\omega,w_{n}+h_{n+1})-\Pi_{W}\partial_{\omega}\mathcal{F}_1(\epsilon,\omega,w_n)\|_{s+\beta}
\leq C(\beta)(1+\|\partial_{\omega} w_n\|_{s})\|h_{n+1}\|_{s+\beta}\\
&\quad\quad\quad\quad\quad\quad\quad\quad\quad\quad\quad\quad\quad\quad\quad\quad\quad\quad\quad\quad\quad\quad\quad\quad+C(\beta)\|w_n\|_{s+\beta}(1+\|\partial_{\omega} w_n\|_{s})\|h_{n+1}\|_{s}\\
&\quad\quad\quad\quad\quad\quad\quad\quad\quad\quad\quad\quad\quad\quad\quad\quad\quad\quad\quad\quad\quad\quad\quad\quad+C(\beta)(1+\|\partial_{\omega} w_n\|_{s+\beta})\|h_{n+1}\|_s.
\grave{}\end{align*}
In addition, if we set $S'_{n}:=1+\|\partial_{\omega}{w_{n}}\|_{s+\beta}$,  we can claim that for $\epsilon{\gamma^{-2}}<\delta_4$ small enough,
\begin{align*}
\textstyle \boldsymbol{\mathrm{(F3)}}:\quad  S'_{n}\leq \tilde{C}_1(1+\bar{K})\gamma^{-1}N^{2\tau+\sigma+{\frac{1}{\chi-1}{(\tau-1+\sigma)}}}_{n+1}
\end{align*}
where $\tilde{C}_1:=\tilde{C}_1(\chi,\mathfrak{q},\tau,\sigma)$ is a positive constant. The proof of $\mathrm{(\bf F3)}$  will be given in  Lemma \ref{lemma4.8}.

Then it follows from \eqref{E3.11}, \eqref{E4.53}, \eqref{E4.57} and \eqref{E4.39} that for ${\epsilon}{\gamma^{-2}}<\delta_4$  small enough,
\begin{align}\label{E3.34}
\|\partial_{\omega}\mathcal{U}_{n+1}(\epsilon,\omega,h_{n+1})\|_{s}
{\leq}&\epsilon \gamma^{-1}C'(\beta)N^{\tau+1}_{n+1}N^{-\beta}_{n}S_{n}+\epsilon C'(\beta)N^{-\beta}_{n}S'_{n},
\end{align}
which then gives
\begin{align*}
\|\partial_{\omega}h_{n+1}\|_s\stackrel{\eqref{E4.19}}{\leq}\epsilon\gamma^{-2} K'N^{2\tau}_{n+1}N^{-\beta}_{n}S_n+\epsilon \gamma^{-1}K'N^{\tau-1}_{n+1}N^{-\beta}_{n}S'_n.
\end{align*}
Furthermore, it can be obtained from \eqref{E4.5} that
\begin{align}\label{E3.29}
N_{n+1}\leq e^{\mathfrak{q}\chi^{n+1}}<N_{n+1}+1<2N_{n+1}.
 \end{align}
Thus using $\mathrm{(\bf F2)}$--$\mathrm{(\bf F3)}$ and \eqref{E5.35} yields that
\begin{align*}
\|\partial_{\omega} h_{n+1}\|_{s}\leq \epsilon\gamma^{-2}K_{4}N_{n+1}^{-1}.
\end{align*}
The proof of the lemma is now accomplished.
\end{proof}

Let us give  the proof of $\mathrm{(\bf F2)}$.

\begin{lemm}\label{lemma4.3}
Let $(\epsilon,\omega)\in A_{n}$. If ${\epsilon}{\gamma^{-2}}<\delta_2\leq\delta_1N^{1-\tau}_0$ is  small enough, then $\mathrm{(\bf F2)}$ holds.
\end{lemm}
\begin{proof}
For ${\epsilon}{\gamma^{-2}}<\delta_2$ small enough, it can be asserted that
\begin{align}\label{E4.36}
S_{n}\leq (1+N^{\tau-1+\sigma}_{n})S_{n-1}.
\end{align}
Then, by virtue of \eqref{E4.5}, \eqref{E3.29} and \eqref{E3.12}, we obtain
\begin{align*}
S_{n}{\leq}&S_0\prod^{n}_{k=1}(1+N^{\tau-1+\sigma}_k)\leq S_0\prod^{n}_{k=1}(1+e^{\mathfrak{q}\chi^k(\tau-1+\sigma)})\\
=&S_0\prod^{n}_{k=1}e^{\mathfrak{q}\chi^k(\tau-1+\sigma)}\prod^{n}_{k=1}(1+e^{-\mathfrak{q}\chi^k(\tau-1+\sigma)})\\
\leq &\tilde{C}(1+\bar{K})N_{n+1}^{\frac{1}{\chi-1}(\tau-1+\sigma)},
\end{align*}
where $\tilde{C}=2^{\frac{1}{\chi-1}
(\tau-1+\sigma)}\prod^{\infty}_{k=1}(1+e^{-\mathfrak{q}\chi^k(\tau-1+\sigma)})$.

It remains to check that \eqref{E4.36} holds. Using the definition of $S_n$ yields that
\begin{align}\label{up}
S_{n}\leq1+\|w_{n-1}\|_{s+\beta}+\|h_{n}\|_{s+\beta}=S_{n-1}+\|h_n\|_{s+\beta}.
\end{align}
Let us verify the upper bound of $\|h_n\|_{s+\beta}$. In view of $\mathrm{(\bf P1)}$--$\mathrm{(\bf P3)}$ and \eqref{E4.33}, we derive
\begin{align*}
&\|r_{n-1}\|_{s}{\leq}\epsilon C,\quad\|R_{n-1}(h_{n})\|_{s}{\leq}\epsilon C\rho^2_n,\quad\|r_{n-1}\|_{s+\beta}\leq \epsilon C(\beta)S_{n-1},\\
&\|R_{n-1}(h_{n})\|_{s+\beta}{\leq}
\epsilon C(\beta)(\rho^2_{n}S_{n-1}+\rho_{n}\|h_n\|_{s+\beta}).
\end{align*}
Since $h_{n}=\mathcal{L}^{-1}_{N_{n}}(\epsilon,\omega,w_{n-1})(R_{n-1}(h_{n})+r_{n-1})$, according to \eqref{E4.12}, \eqref{E4.52} and \eqref{E4.33}, one has
\begin{align*}
\|h_n\|_{s+\beta}{\leq}{\epsilon\gamma^{-1} K'}N^{\tau-1+\sigma}_{n}S_{n-1}+{\epsilon \gamma^{-1}K'}N^{\tau-1}_{n}\rho_n\|h_n\|_{s+\beta}.
\end{align*}
From $\tau\in(1,2)$ and the definition of $\sigma$ (recall \eqref{E0.17}), it follows that  $\sigma>\tau-1$. Then there exists $\delta_2>0$ small enough with ${\epsilon}{\gamma^{-1}}\leq{\epsilon}{\gamma^{-2}}<\delta_2$ such that ${\epsilon \gamma^{-1}K'}\leq\frac{1}{2}$, which implies that
\begin{align}\label{E4.17}
\|h_n\|_{s+\beta}\leq{2\epsilon\gamma^{-1} K'}N^{\tau-1+\sigma}_{n}S_{n-1}\leq N^{\tau-1+\sigma}_{n}S_{n-1}.
\end{align}
Substituting this into \eqref{up} gives rise to \eqref{E4.36}. Hence we arrive at the conclusion of the lemma.
\end{proof}

In order to verify the fact $\mathrm{(\bf F3)}$, we need to estimate the upper bound of $\mathcal{L}^{-1}_{N_{n}}(\epsilon,\omega,w_{n})\mathfrak{w}$ on $(s+\beta)$-norm for all $\mathfrak{w}\in W_{N_{n}}$, where $\mathcal{L}_{N_{n}}(\epsilon,\omega,w_{n})$ is as seen in  \eqref{E8.5}.

\begin{lemm}
For $(\epsilon,\omega)\in A_{n}$ and $\epsilon{\gamma^{-2}}<\delta_4\leq\delta_3$ small enough, one has
\begin{align}\label{E:inversion}
\|\mathcal{L}^{-1}_{N_{n}}(\epsilon,\omega,w_{n})\mathfrak{w}\|_{s+\beta}\leq&{K_5}{\gamma^{-1}}N^{\tau-1}_{n}\|\mathfrak{w}\|_{s+\beta}\nonumber\\
&+{K_5}{\gamma^{-1}}N^{2(\tau-1)}_{n}(\|w_{n-1}\|_{s+\beta+\sigma}+\|h_{n}\|_{s+\beta})\|\mathfrak{w}\|_{s}
\end{align}
for some constant $K_5>0$, where  $\beta,\sigma$ are given  by \eqref{E5.35}, \eqref{E0.17}, respectively.
\end{lemm}
\begin{proof}
If we denote $\mathcal{G}(h_{n}):=(\mathrm{Id}-\mathrm{D}U_{n}(h_{n}))^{-1}\mathfrak{w}$, then it is evident that
\begin{align*}
\mathcal{G}(h_{n})=\mathfrak{w}+\mathrm{D}U_{n}(h_{n})\mathcal{G}(h_{n}).
\end{align*}
Applying \eqref{E4.12}, \eqref{E4.15} and $\mathrm{(\bf P2)}$ reads that
\begin{align*}
\|\mathrm{D}U_{n}(h_{n})\|_{s+\beta}\leq
{\epsilon \gamma^{-1}K'}N^{\tau-1}_{n}(\|w_{n-1}\|_{s+\beta+\sigma}\|h_{n}\|_{s}+\|h_{n}\|_{s+\beta}).
\end{align*}
Furthermore, it can be obtained from \eqref{E4.16} that $\|\mathcal{G}(h_{n})\|_s\leq2\|\mathfrak{w}\|_s$. Then combining this  with \eqref{E2.1}, \eqref{E4.16} and \eqref{E4.33} yields that
\begin{align*}
\|\mathcal{G}(h_{n})\|_{s+\beta}
{\leq}&\|\mathfrak{w}\|_{s+\beta}+2{\epsilon \gamma^{-1} {K}'}N^{\tau-1}_{n}(\|w_{n-1}\|_{s+\beta+\sigma}\|h_{n}\|_{s}+\|h_{n}\|_{s+\beta})\|\mathfrak{w}\|_{s}\\
&+{\epsilon \gamma ^{-1}K'}N^{\tau-1}_{n}\rho_{n}\|\mathcal{G}(h_{n})\|_{s+\beta}.
\end{align*}
By the fact $\tau\in(1,2)$, the definition of $\sigma$ shows that  $\sigma>\tau-1$. Owing to both this and \eqref{E4.33}, we can deduce that for $\epsilon{\gamma^{-1}}\leq\epsilon{\gamma^{-2}}<\delta_4$ small enough,
\begin{align*}
\|\mathcal{G}(h_{n})\|_{s+\beta}\leq&2\|\mathfrak{w}\|_{s+\beta}+{4\epsilon \gamma^{-1} {K}'}N^{\tau-1}_{n}(\|w_{n-1}\|_{s+\beta+\sigma}\|h_{n}\|_{s}+\|h_{n}\|_{s+\beta})\|\mathfrak{w}\|_{s}.
\end{align*}
If $\epsilon{\gamma^{-2}}<\delta_4$ is small enough,  then \eqref{E:inversion} follows from \eqref{E4.13}--\eqref{E4.12}. Thus the proof of the lemma is now completed.
\end{proof}

It remains to check the fact $\mathrm{(\bf F3)}$.

\begin{lemm}\label{lemma4.8}
Let $(\epsilon,\omega)\in A_{n}$. If ${\epsilon}{\gamma^{-2}}<\delta_4\leq\delta_3$ is mall enough, then $\mathrm{(\bf F3)}$ holds.
\end{lemm}
\begin{proof}
For ${\epsilon}{\gamma^{-2}}<\delta_4$  small enough,  let us claim that for  some constant $K'>0$,
\begin{align}
S'_{n}\leq(1+N^{\tau-1}_{n})S'_{n-1}+{K'}{\gamma^{-1}}N^{2\tau+\sigma}_{n}S_{n-1}.\label{E4.40}
\end{align}
If we take $\alpha_1:=\tau-1$, $\alpha_2:=2\tau+\sigma$, $\alpha_3:=\tau-1+\sigma$, then a simple calculation gives that
\begin{align*}
\textstyle S'_{n}\leq \mathcal{S}_1+\mathcal{S}_2 ,\text{ with }\mathcal{S}_1=S'_{0}\prod^{n}_{k=1}\left(1+N^{\alpha_1}_{k}\right),~\mathcal{S}_2=\sum^n_{k=1}\mathcal{S}_{2,k},
\end{align*}
where $\mathcal{S}_{2,1}={K'}{\gamma^{-1}}N^{\alpha_2}_{n}S_{n-1}$ and for $2\leq k\leq n$,
\begin{align*}
\mathcal{S}_{2,k}={K'}{\gamma^{-1}}\prod_{j=2}^k(1+N^{\alpha_1}_{n-(j-2)})N^{\alpha_2}_{n-(k-1)}S_{n-k}.
\end{align*}
Since the upper bound on $\mathcal{S}_1$ is proved in the similar way as shown in Lemma \ref{lemma4.3},  the detail is omitted. Hence using \eqref{E3.28} yields that
\begin{align*}
\mathcal{S}_1\leq{C}'_1 (1+\bar{K})\gamma^{-1}N^{\frac{1}{\chi-1}\alpha_1}_{n+1}.
\end{align*}
On the one hand, it follows from \eqref{E4.5} and $\mathrm{(\bf F2)}$ that
\begin{align*}
\mathcal{S}_{2,1}\leq K'\gamma^{-1}N^{\alpha_2}_{n}\tilde{C}(1+\bar{K})N_{n}^{\frac{1}{\chi-1}\alpha_3}
\leq C'_1(1+\bar{K})\gamma^{-1}N^{\alpha_2+\frac{1}{\chi-1}\alpha_3}_{n+1}.
\end{align*}
And on the other hand, a simple computation yields that
\begin{align*}
\sum\limits_{k=2}^{n} \mathcal{S}_{2,k}&\leq K'S_0\gamma^{-1}\sum\limits_{k=2}^{n}e^{\alpha_1\mathfrak{q}\frac{\chi^{n+1}-\chi^{n+2-k}}{\chi-1}}
e^{\alpha_2\mathfrak{q}\chi^{n+1-k}}e^{{\alpha_3}\mathfrak{q}\frac{\chi^{n+1-k}}{\chi-1}}\nonumber\\
&\leq\frac{K'S_0}{\gamma\omega}e^{\alpha_1\mathfrak{q}\frac{\chi^{n+1}}{\chi-1}}
\sum\limits_{k=2}^{n}e^{{(-\alpha_1+(\chi-1)\alpha_2+\alpha_3)}\mathfrak{q}\frac{\chi^{n+2-k}}{\chi-1}}\nonumber\\
&\leq C'_1(1+\bar{K})\gamma^{-1} N^{\alpha_2+\frac{1}{\chi-1}\alpha_3}_{n+1}.
\end{align*}

The remainder of the lemma is to prove \eqref{E4.40}. It follows from the definition of $S'_{n}$ that
\begin{align*}
S'_{n}\leq1+\|\partial_{\omega}w_{n-1}\|_{s+\beta}+\|\partial_{\omega}h_n\|_{s+\beta}
=S'_{n-1}+\|\partial_{\omega}h_n\|_{s+\beta}.
\end{align*}
We just establish the upper bound of $\partial_{\omega}h_n$ on $(s+\beta)$-norm. Formulae \eqref{E4.20} and \eqref{E:inversion} read that
\begin{align*}
\|\partial_{\omega}h_{n}\|_{s+\beta}\leq&{K_5}{\gamma^{-1}}N^{\tau-1}_{n}\|\partial_{\omega}\mathcal{U}_{n}(\epsilon,\omega,h_n)\|_{s+\beta}\nonumber\\
&+{ K_5}{\gamma^{-1}}N^{2(\tau-1)}_{n}(\|w_{n-1}\|_{s+\beta+\sigma}+\|h_{n}\|_{s+\beta})\|\partial_{\omega}\mathcal{U}_{n}(\epsilon,\omega,h_n)\|_{s}.
\end{align*}
Let us estimate the upper bound of $\|(h_{n})_{tt}\|_{s+\beta}$. For ${\epsilon}{\gamma^{-2}}<\delta_4$  small enough,
it follows  from \eqref{E4.56}, \eqref{E4.52}, \eqref{E4.17}, $\mathrm{(\bf P2)}$ and \eqref{E4.33} that
\begin{align*}
\|(h_{n})_{tt}\|_{s+\beta}\stackrel{}\leq{C'(\beta)}{\omega^{-2}}N^{\tau+1+\sigma}_{n}S_{n-1}.
\end{align*}
Then for ${\epsilon}{\gamma^{-2}}<\delta_4$  small enough, it can be seen from  \eqref{E3.31}, \eqref{E4.39} and \eqref{E4.17} that
\begin{align*}
\|\partial_{\omega}\mathcal{U}_{n}(\epsilon,\omega,h_n)\|_{s+\beta}\leq {C''(\beta)}N^{\tau+1+\sigma}_{n}S_{n-1}+\epsilon C''(\beta)S'_{n-1}.
\end{align*}
According to \eqref{E3.34},  \eqref{E5.35} and $\mathrm{(\bf F2)}$--$\mathrm{(\bf F3)}$, it is straightforward that
\begin{align*}
\|\partial_{\omega}\mathcal{U}_{n+1}(\epsilon,\omega,h_{n+1})\|_{s}\leq 2\epsilon \gamma^{-1}C'(\beta).
\end{align*}
Consequently, by means of  \eqref{E4.17}, we derive
\begin{align*}
\|\partial_{\omega}h_{n}\|_{s+\beta}\leq&{K'}{\gamma^{-1}}N^{2\tau+\sigma}_{n}S_{n-1}+N^{\tau-1}_{n}S'_{n-1}
\end{align*}
for ${\epsilon}{\gamma^{-2}}<\delta_4$  small enough. This gives rise to \eqref{E4.40}. Thus this ends the proof.
\end{proof}

\section{Proof of the main result}\label{sec:proof}
In this section, we  are devoted to accomplishing the proof of the main result given by Theorem \ref{Th1}. Our first purpose is to present that $(\epsilon,\omega) $ belong to a Cantor set of positive measure, asymptotically full as $\epsilon$ tends to zero.
 From Lemma \ref{lemma4.1}, define
\begin{align}\label{E0.2}
A_\gamma:=\{(\epsilon,\omega)\in A_n, \forall n\in\mathbb{N}, \epsilon<\delta\gamma^2\},
\end{align}
where $\delta\in(0,\delta_4]$. Let us establish the measure estimate on $A_\gamma$.

\begin{lemm}\label{lemma4}
Let $I:=(\tilde{\omega}_1,\tilde{\omega}_2)$ with $\frac{1}{2}\leq\tilde{\omega}_1<\tilde{\omega}_2\leq\frac{3}{2}$. There exists some constant $\delta\in(0,\delta_4]$ such that for all $\gamma\in(0,1)$,  the Lebesgue measures of the Cantor set $A_\gamma$
and its sections $A_{\gamma}(\epsilon):=\{\omega:(\epsilon,\omega)\in A_\gamma\}$ have the following property: there exists some constant $\mathrm{C}>0$, independent on $\gamma $ and $\epsilon$, such that for all $\epsilon<\delta\gamma^2$,
\begin{align}\label{E0.10}
\frac{\mathrm{meas}(I\cap A_{\gamma}(\epsilon))}{\mathrm{meas}(I)}>1-\mathrm{C}\gamma,\quad \frac{\mathrm{meas(B_\gamma\cap A_{\gamma})}}{\mathrm{meas}(B_\gamma)}>1-\mathrm{C}\gamma,
\end{align}
where $B_\gamma$ is the rectangular region $(0,\delta\gamma^2)\times I$.
\end{lemm}
\begin{proof}
For fixed $\epsilon\in(\epsilon_1,\epsilon_2)$, we define
\begin{align*}
\mathrm{R}_{n}=A_n(\epsilon)\backslash A_{n+1}(\epsilon), \quad\forall n\in\mathbb{N}.
\end{align*}
Let us present that $\bigcup_{n\in\mathbb{N}} \mathrm{R}_{n}\cap (\tilde{\omega}_1,\tilde{\omega}_2)$ is a small measure set. Denote
\begin{align*}
&\textstyle\Omega^{n,1}_{j,l}:=\left\{\omega\in(\frac{1}{2},\frac{3}{2}):|\omega(\lambda^{\pm}_{l}(\epsilon,\omega,w_n))^{\frac12}-j|\leq\frac{\gamma}{j^{\tau}}\right\},
\quad\textstyle\Omega^{n,2}_{j,l}:=\left\{\omega\in(\frac{1}{2},\frac{3}{2}):|\frac{\omega l}{c}-j|\leq\frac{\gamma}{j^{\tau}}\right\}.
\end{align*}
The fact $\gamma<1$ implies that $\Omega^{n,2}_{j,0}=\emptyset$. Moreover, if we denote $\tilde{A}_{0}:=A_0$ and for all $n\in\mathbb{N}^{+}$,
\begin{align*}
\textstyle\tilde{A}_{n}:=\left\{(\epsilon,\omega)\in \tilde{A}_{n-1}:|\omega(\lambda^{\pm}_{l}(\epsilon,\omega,w_{n-1}))^{\frac12}-j|>\frac{\gamma}{j^{\tau}},\forall 1\leq j\leq N_n, \forall l\geq0\right\},
\end{align*}
then $\mathrm{R}_n\subseteq \mathrm{R}_{n,1} \cup\mathrm{R}_{n,2}$, where
\begin{align*}
\mathrm{R}_{n,1}=&\textstyle\bigcup_{1\leq j\leq N_{n+1},l\geq 0}\Omega^{n,1}_{j,l}\cap \tilde{A}_{n}(\epsilon)\cap (\tilde{\omega}_1,\tilde{\omega}_2),\\
\mathrm{R}_{n,2}=&\textstyle\bigcup_{N_n< j\leq N_{n+1},l\geq1}\Omega^{n,2}_{j,l}\cap (\tilde{\omega}_1,\tilde{\omega}_2),
\end{align*}

First, we need to  estimate the measure of $\mathrm{R}_{n,1}$. In view of \eqref{E6.28}, it follows that for some constants $C_1>0,C_2>0$,
\begin{align} \label{E0.13}
C_1 j<\omega l< C_2 j, \quad \forall l\geq  \max\{J_0,\lfloor(2\Theta c)^{\frac12}\rfloor+1\}+1,\forall \omega \in\Omega^{n,1}_{j,l}.
\end{align}
In particular, observe that
\begin{align*}
C_1 j<\omega(\lambda^{\pm}_l(\epsilon,\omega,w_n))^{\frac12}<C_2 j,\quad 0\leq l\leq \max\{J_0,\lfloor(2\Theta c)^{\frac12}\rfloor+1\}.
\end{align*}
Moreover, Lemma \ref{lemma4.1} shows that
\begin{align*}
\textstyle\|w_n\|_{s}\leq\sum_{k=0}^{n}\|h_{k}\|_{s}
\stackrel{\eqref{E4.11}}{\leq}\sum_{k=1}^{n}{K_2\epsilon \gamma^{-1}N^{-\sigma-1}_{k}}+{ K_1\epsilon\gamma^{-1} N^{\tau-1}_0}< r.
\end{align*}
Hence it can be seen that $|\lambda^{\pm}_l(\epsilon,\omega,w_n)|\geq \Gamma_0>0$ (recall \eqref{E:min}). As a result, combining this with formula \eqref{E:continue2} yields that for all $\omega_1,\omega_2\in \tilde{A}_{n}(\epsilon)$ with $\omega_1<\omega_2$,
\begin{align}\label{E0.5}
&|(\lambda^{\pm}_{l}(\epsilon,\omega_2,w_n(\epsilon,\omega_2)))^{\frac12}
-(\lambda^{\pm}_{l}(\epsilon,\omega_1,w_n(\epsilon,\omega_1)))^{\frac12}|\nonumber\\
&=\frac{|\lambda_{l}^{\pm}(\epsilon,\omega_2,w_n(\epsilon,\omega_{2}))-\lambda^{\pm}_{l}(\epsilon,\omega_1,w_n(\epsilon,\omega_1))|}
{(\lambda^{\pm}_{l}(\epsilon,\omega_2,w_n(\epsilon,\omega_2)))^{\frac12}+
(\lambda^{\pm}_{l}(\epsilon,\omega_1,w_n(\epsilon,\omega_1)))^{\frac12}}\nonumber\\
&\leq\frac{\epsilon\kappa}{\sqrt{\Gamma_0}}|\omega_1-\omega_2|
+\frac{\epsilon\kappa}{\sqrt{\Gamma_0}}\|w_n(\epsilon,\omega_1)-w_n(\epsilon,\omega_2)\|_s.
\end{align}
For ${\epsilon}{\gamma^{-1}}\leq{\epsilon}{\gamma^{-2}}<\delta$ small enough, it follows from  \eqref{E6.28} and \eqref{E4.39} that
\begin{align}\label{E0.3}
&|\omega_2(\lambda^{\pm}_{l}(\epsilon,\omega_2,w_n(\epsilon,\omega_2)))^{\frac12}
-\omega_1(\lambda^{\pm}_{l}(\epsilon,\omega_1,w_n(\epsilon,\omega_1)))^{\frac12}|
\nonumber\\
&\geq\textstyle|\omega_2-\omega_1|\frac{l}{2c}-\frac{\epsilon\kappa}{\sqrt{\Gamma_0}}|\omega_1-\omega_2|
-\frac{\epsilon\kappa}{\sqrt{\Gamma_0}}\frac{\bar{{K}}\epsilon}{\gamma^2}|\omega_1-\omega_2|\nonumber\\
&\geq\textstyle |\omega_2-\omega_1|\frac{l}{4c},\quad\quad \quad\forall l\geq \max\{J_0,\lfloor(2\Theta c)^{\frac12}\rfloor+1\}+1.
\end{align}
In particular, if ${\epsilon}{\gamma^{-2}}<\delta$ is small enough, then for $0\leq l\leq \max\{J_0,\lfloor(2\Theta c)^{\frac12}\rfloor+1\}$,
\begin{align}\label{E0.15}
\textstyle|\omega_2(\lambda^{\pm}_{l}(\epsilon,\omega_2,w_n(\epsilon,\omega_2)))^{\frac12}
-\omega_1(\lambda^{\pm}_{l}(\epsilon,\omega_1,w_n(\epsilon,\omega_1)))^{\frac12}|\geq|\omega_2-\omega_1|\frac{\sqrt{\Gamma_0}}{4}.
\end{align}
Consequently, if $\Omega^{n,1}_{j,l}\cap \tilde{A}_{n}(\epsilon)\cap(\tilde{\omega}_1,\tilde{\omega}_2)$ is nonempty, then using \eqref{E0.3}--\eqref{E0.15} and \eqref{E0.13} yields that
\begin{align}
&\textstyle\mathrm{meas}(\Omega^{n,1}_{j,l})\leq\frac{8\gamma}{\sqrt{\Gamma_0}j^{\tau}},\quad 0\leq l\leq\max\{J_0,\lfloor(2\Theta c)^{\frac12}\rfloor+1\},\label{new2}\\
&\textstyle\mathrm{meas}(\Omega^{n,1}_{j,l})\leq\frac{8c\gamma}{lj^{\tau}}\leq\frac{C\tilde{\omega}_2\gamma}{j^{\tau+1}},\quad \forall l\geq \max\{J_0,\lfloor(2\Theta c)^{\frac12}\rfloor+1\}+1,l\in\left(\frac{C_1}{\tilde{\omega}_2}j,\frac{C_2}{\tilde{\omega}_1}j\right)=:\mathcal{I}(j),j\geq1.\label{E0.4}
\end{align}
Because of \eqref{new2}--\eqref{E0.4}, we establish that
\begin{align*}
\textstyle\mathrm{meas}\left(\bigcup_{1\leq j\leq N_1,l\geq0}\Omega^{n,1}_{j,l}\cap (\tilde{\omega}_1,\tilde{\omega}_2)\right){\leq}\mathrm{C}_1\gamma\sum_{1\leq j\leq N_1}\frac{1}{j^{\tau}}.
\end{align*}
This implies that
\begin{align}\label{zero}
\textstyle\mathrm{meas}\left(\mathrm{R}_{0,1}\right){\leq}\mathrm{C}_1\gamma\sum_{1\leq j\leq N_1}\frac{1}{j^{\tau}}.
\end{align}
Our next goal is to give the measure of $\mathrm{R}_{n,1}$ for $n\in\mathbb{N}^{+}$. Let us consider either $N_n<j\leq N_{n+1}$ or $1\leq j\leq N_n$. In the first case, it follows from \eqref{new2}--\eqref{E0.4} that
\begin{align*}
\textstyle\mathrm{meas}\left(\bigcup_{N_n<j\leq N_{n+1},l\geq0}\Omega^{n,1}_{j,l}\cap \tilde{A}_n(\epsilon)\cap (\tilde{\omega}_1,\tilde{\omega}_2)\right)\leq\mathrm{C}_1\gamma\sum_{N_n<j\leq N_{n+1}}\frac{1}{j^{\tau}}.
\end{align*}
And in the latter, formulae \eqref{E0.5} and \eqref{E3.11} read that
\begin{align*}
\textstyle|\omega(\lambda^{\pm}_{l}(\epsilon,\omega,w_{n-1}))^{\frac12}-j|\leq\frac{\gamma}{j^{\tau}}+\frac{\epsilon\kappa}{\sqrt{\Gamma_0}}
\|w_n-w_{n-1}\|_s\leq\frac{\gamma}{j^{\tau}}+\frac{\epsilon^2 \mathrm{C}'}{\gamma}N^{\tau-1}_{n}N^{-\beta}_{n-1}S_{n-1}.
\end{align*}
It is easy to see that for $1\leq j\leq N_n$,
\begin{align*}
\textstyle\Omega^{n,1}_{j,l}\cap \tilde{A}_n(\epsilon)\subseteq\left\{\omega:\frac{\gamma}{j^{\tau}}<|\omega(\lambda^{\pm}_{l}(\epsilon,\omega,w_{n-1}))^{\frac12}-j|
\leq\frac{\gamma}{j^{\tau}}+\frac{\epsilon^2\mathrm{C}'}{\gamma}N^{\tau-1}_{n}N^{-\beta}_{n-1}S_{n-1}\right\}.
\end{align*}
Consequently, it can be obtained from \eqref{E:min} that
\begin{align*}
\textstyle\mathrm{meas}\left(\Omega^{n,1}_{j,l}\cap \tilde{A}_n(\epsilon)\cap (\tilde{\omega}_1,\tilde{\omega}_2)\right)\leq\frac{\epsilon^2 \mathrm{C}''}{\gamma \sqrt{\Gamma_0}}N^{\tau-1}_{n}N^{-\beta}_{n-1}S_{n-1}.
\end{align*}
From the fact $\mathrm{(\bf F2)}$, \eqref{new2}--\eqref{E0.4} and \eqref{E5.35}, one has that for $1\leq j\leq N_n$,
\begin{align*}
\textstyle \mathrm{meas}\left(\bigcup_{1\leq j\leq N_n,l\geq0}\Omega^{n,1}_{j,l}\cap \tilde{A}_n(\epsilon)\cap (\tilde{\omega}_1,\tilde{\omega}_2)\right)\leq\sum_{1\leq j\leq N_n}\frac{\epsilon^2 \mathrm{C}'''j}{\sqrt{\Gamma_0}\gamma}N^{\tau-1}_{n}N^{-\beta}_{n-1}S_{n-1}
\leq\mathrm{C}_1\gamma N^{-1}_{n}.
\end{align*}
Therefore we can conclude that for all $n\in\mathbb{N}^+$,
\begin{align}\label{new3}
\textstyle\mathrm{meas}(\mathrm{R}_{n,1})\leq\mathrm{C}_1\gamma\sum_{N_{n}< j\leq N_{n+1}}\frac{1}{j^{\tau}}+\mathrm{C}_1\gamma N^{-1}_{n}.
\end{align}

Another step is to estimate the measure of $\mathrm{R}_{n,2}$. It is obvious that
\begin{align*}
 C_1 j<\omega l< C_2 j,\quad\forall l\in\mathbb{N}^{+},\forall\omega \in\Omega^{n,2}_{j,l}.
\end{align*}
Applying the similar  technique as above yields that $\mathrm{meas}(\Omega^{n,2}_{j,l})<\frac{C\tilde{\omega}_2\gamma}{j^{\tau+1}}$, $\forall l\in\mathcal{I}(j)$. This leads to
\begin{align}\label{E0.11}
\textstyle\mathrm{meas}\left(\bigcup_{N_n< j\leq N_{n+1},l\geq1}\Omega^{n,2}_{j,l}\cap(\tilde{\omega}_1,\tilde{\omega}_2)\right){\leq}\mathrm{C}_1\gamma\sum_{N_n<j\leq N_{n+1}}\frac{1}{j^{\tau}},\quad \forall n\in\mathbb{N}.
\end{align}
As a result, by virtue of \eqref{zero}--\eqref{E0.11} and the fact $\tau>1$, we get
\begin{align*}
\textstyle\mathrm{meas}\left(\bigcup_{n\in\mathbb{N}} \mathrm{R}_{n}\cap (\tilde{\omega}_1,\tilde{\omega}_2)\right)\leq \mathrm{C}\gamma\left(\sum_{n\in \mathbb{N}^{+}}N^{-1}_{n}+\sum_{j\geq1}\frac{1}{j^{\tau}}\right)=O(\gamma).
\end{align*}
Notice that the second term in \eqref{E0.10} follows from
\begin{align*}
\textstyle\mathrm{meas(B_\gamma\cap A_{\gamma})}=\int^{\delta \gamma^2}_{0}\mathrm{meas}(A_\gamma(\epsilon)\cap(\tilde{\omega}_1,\tilde{\omega}_2))\mathrm{d}\epsilon.
\end{align*}
Hence the proof of this lemma is completed.
\end{proof}

Theorem \ref{Th1} follows from Lemma \ref{lemma3.1}, Lemma \ref{lemma4.1} and  Lemma \ref{lemma4}. We are now turning to  the proof of Theorem \ref{Th1}.
\begin{proof}[\underline{Proof of Theorem \ref{Th1}}]
It is clear that $A_{\gamma}\subset (0,\delta\gamma^2)\times(\frac12,\frac23)$. Lemma \ref{lemma4} shows that $A_\gamma$ is a Cantor set. For all $(\epsilon,\omega)\in A_{\gamma}$, it follows from Lemma \ref{lemma4.1} that $w(\epsilon,\omega)\in C^1(A_\gamma;W\cap H^s)$ solves  the range equation \eqref{E:range}.  Moreover, in view of Lemma \ref{lemma4.1}, we can derive
\begin{align*}
\textstyle\|w\|_{s}\leq\sum_{k=0}^{+\infty}\|h_{k}\|_{s}
\stackrel{\eqref{E4.11}}{\leq}\sum_{k=1}^{+\infty}{K_2\epsilon \gamma^{-1}N^{-\sigma-1}_{k}}+{ K_1\epsilon\gamma^{-1} N^{\tau-1}_0}< r.
\end{align*}
Then it can be seen from Lemma \ref{lemma3.1} that $v(\epsilon,\omega,w)$ solves the bifurcation
equation \eqref{E:bifurcation}. As a result,
\begin{align*}
{u}(\epsilon,\omega)=v(\epsilon,\omega,{w}(\epsilon,\omega))+{w}(\epsilon,\omega)\in {H}^1(\mathbb{T})\oplus(W\cap H^s),
\end{align*}
with $\int_{\mathbb{T}^2}u(\epsilon,\omega)(t, x)\mathrm{d}t\mathrm{d}x=0$,
is a solution of equation \eqref{E1.2}.

In addition, since $u$ solves
\begin{align*}
\omega^2u_{tt}(\cdot,x)=\epsilon f(x,u(\cdot,x))+\epsilon g(\cdot,x)+a(\cdot)u_{xx}(\cdot,x),
\end{align*}
we have
\begin{align*}
u_{tt}(\cdot,x)\in H^{1}(\mathbb{T}),\quad\forall x\in\mathbb{T}.
\end{align*}
Hence we have $u(\cdot,x)\in H^{3}{(\mathbb{T})}\subset C^2(\mathbb{T})$ for all $x\in\mathbb{T}$.

Consequently, we obtain  the conclusion of Theorem \ref{Th1}.
\end{proof}

%

\section{Appendix}
In the Appendix, we will supplement some results used in the  proof of Theorem \ref{Th1} for completeness. First, we need to give the proof of Remark \ref{remark1}.

\begin{proof}[\underline{Proof of Remark \ref{remark1}}]
{\rm{(i)}} For all $u,v\in H^s$, we decompose $uv=\sum_{j}(\sum_{k}u_{j-k}v_{k})e^{\mathrm{i}jx}$.
Then using the Cauchy inequality yields that
\begin{align*}
\|uv\|^2_{s}=&\textstyle\sum_{j}(1+j^{2s})\|\sum_{k}u_{j-k}v_{k}\|^2_{H^1}
\leq\sum_{j}(\sum_{k}(1+j^{2s})^{\frac{1}{2}}c_{jk}\|u_{j-k}v_{k}\|_{H^1}\frac{1}{c_{jk}})^2\\
\leq&\textstyle\sum_{j}(\sum_{k}\frac{1}{c^2_{jk}})(
\sum_{k}\|u_{j-k}\|^2_{H^1}(1+(j-k)^{2s})\|v_{k}\|^2_{H^1}(1+k^{2s})),
\end{align*}
where $c^2_{jk}=\frac{(1+k^{2s})(1+(j-k)^{2s})}{1+j^{2s}}$. Moreover, it is easy to obtain that
\begin{align*}
1+j^{2s}&\leq1+(k+j-k)^{2s}\leq1+2^{2s-1}(k^{2s}+(j-k)^{2s})\\
&\leq2^{2s-1}(1+k^{2s})+2^{2s-1}(1+(j-k)^{2s}).
\end{align*}
This leads to
\begin{align*}
\textstyle\sum_{k}\frac{1}{c^2_{jk}}\leq2^{2s-1}(\sum_{k}\frac{1}{1+k^{2s}}+\sum_{k}\frac{1}{1+(j-k)^{2s}})
=2^{2s}\sum_{k}\frac{1}{1+k^{2s}}\stackrel{s>1/2}{=:}(C(s))^2.
\end{align*}
Hence we give rise to the first conclusion in Remark \ref{remark1}.

{\rm{(ii)}} From the  Cauchy inequality, we get
\begin{align*}
\textstyle \max_{x\in \mathbb{T}}\|u(\cdot,x)\|_{{H}^1(\mathbb{T})}\leq C\sum_{j}\|u_j\|_{H^1}\leq C (\sum_{j}\|u_j\|^2_{H^1}(1+j^{2s}))^{\frac{1}{2}}(\sum_{j}\frac{1}{1+j^{2s}})^{\frac{1}{2}}\leq C(s)\|u\|_{s}.
\end{align*}
Thus we arrive at the other one in Remark \ref{remark1}.
\end{proof}

The next thing is to recall Lemmata \ref{lemma2.1}--\ref{lemma2.5} introduced in \cite[cf. Lemmata 2.1--2.3]{berti2008cantor}.
\begin{lemm}[Moser--Nirenberg]\label{lemma2.1}
Let $s'\geq0$ and $s>1/2$. For all $u_1,u_2\in H^{s'}\cap H^s$, one has
\begin{align}
\|u_1u_2\|_{s'}
&\leq\textstyle  C(s')(\|u_1\|_{L^{\infty}(\mathbb{T};H^1(\mathbb{T}))}\|u_2\|_{{s'}}+\|u_1\|_{{s'}}\|u_2\|_{L^{\infty}(\mathbb{T};H^1(\mathbb{T}))})\label{E2.10}\\
&\leq  C(s')\left(\|u_1\|_{s}\|u_2\|_{{s'}}+\|u_1\|_{{s'}}\|u_2\|_{s}\right).\label{E2.1}
\end{align}
\end{lemm}
\begin{lemm}[Logarithmic convexity]\label{lemma2.2}
Let $0\leq\mathrm{a}\leq \mathrm{a}'\leq \mathrm{b}'\leq\mathrm{b}$ with $\mathrm{a}+\mathrm{b}=\mathrm{a}'+\mathrm{b}'$, and $\mathfrak{v}:=\frac{\mathrm{b}-\mathrm{a}'}{\mathrm{b}-\mathrm{a}}$. Then
\begin{align*}
\textstyle \|u_1\|_{\mathrm{a}'}\|u_2\|_{\mathrm{b}'}\leq\mathfrak{v}\|u_1\|_{\mathrm{a}}\|u_2\|_{\mathrm{b}}
+(1-\mathfrak{v})\|u_2\|_{\mathrm{a}}\|u_1\|_{\mathrm{b}},\quad u_1,u_2\in H^{\mathrm{b}}.
\end{align*}
In particular, one has
\begin{align}\label{E2.3}
\|u\|_{\mathrm{a}'}\|u\|_{\mathrm{b}'}\leq\|u\|_{\mathrm{a}}\|u\|_{\mathrm{b}},\quad u\in H^{\mathrm{b}}.
\end{align}
\end{lemm}
\begin{lemm}\label{lemma2.5}
Let $y\longmapsto f(\cdot,y)$ be in $C^1(\mathbb{R};H^{1}(\mathbb{T}))$ and $m:=\|y\|_{L^{\infty}(\mathbb{T})}$. Then the composition operator $y(t)\longmapsto f(t,y(t))$ belongs to $C(H^1(\mathbb{T});H^1(\mathbb{T}))$ satisfying
\begin{align*}
\textstyle\|f(\cdot,y)\|_{H^1}\leq C(\max_{y\in[-m,m]}\|f(\cdot,y)\|_{H^1}+\max_{y\in[-m,m]}\|\partial_yf(\cdot,y)\|_{H^1}\|y\|_{H^1}).
\end{align*}
\end{lemm}

Based on Lemmata \ref{lemma2.1}--\ref{lemma2.5}, we have the following lemma.

\begin{lemm}\label{lemma2.3}
Let $(x,u)\longmapsto f(\cdot,x,u)$ belong to $C^\ell(\mathbb{T}\times\mathbb{R};H^1(\mathbb{T}))$ with $\ell\geq1$. For all $s>{1}/{2}$ and $0\leq s'\leq \ell-1$, the composition operator
$u(t,x)\longmapsto f(t,x,u(t,x))$ is in $C(H^{s}\cap H^{s'};H^{s'})$ with
\begin{align*}
\|f(t,x,u)\|_{{s'}}\leq C(s',\|u\|_{s})(1+\|u\|_{{s'}}).
\end{align*}
\end{lemm}
\begin{proof}
If $s'=\mathrm{p}\in\mathbb{N}$ with $\mathrm{p}\leq \ell-1$, we need to prove that for all $u\in H^{s}\cap H^{\mathrm{p}}$,
\begin{align}\label{E3.0}
\|f(t,x,u)\|_{\mathrm{p}}\leq C(\mathrm{p},\|u\|_{s})(1+\|u\|_{\mathrm{p}}),
\end{align}
 and that
 \begin{align}\label{E3.1}
 f(t,x,u_n)\rightarrow f(t,x,u)\quad \text{as } u_n\rightarrow u ~\text{in }  H^{s}\cap H^{\mathrm{p}}.
 \end{align}
 Let us verify \eqref{E3.0}--\eqref{E3.1} by a recursive argument.

For $\mathrm{p}=0$, it follows from Lemma \ref{lemma2.5} and Remark \ref{remark1} that
\begin{align}
\textstyle\|f(t,x,u)\|_{0}\leq& \textstyle C\max_{x\in\mathbb{T}}\|f(\cdot,x,u(\cdot,x))\|_{H^1(\mathbb{T})}\nonumber\\
{\leq}&\textstyle C'(1+\max_{x\in\mathbb{T}}\|u(\cdot ,x)\|_{H^1(\mathbb{T})})\nonumber\\
{\leq}&\textstyle C''(1+\|u\|_{s})=:C(\|u\|_{s}).\label{E2.00}
\end{align}
If $\ell\geq2$, then a similar argument as above can yield that
\begin{align}\label{E2.7}
\textstyle \|\partial_{x}f(t,x,u)\|_{0}\leq C(\|u\|_{s}),\quad\max_{x\in\mathbb{T}}\|\partial_{u}f(\cdot,x,u(\cdot,x))\|_{H^1(\mathbb{T})}\leq C(\|u\|_{s}).
\end{align}
Moreover, it can be shown from Remark \ref{remark1} that
\begin{align*}
\textstyle\max_{x\in\mathbb{T}}\|u_{n}(\cdot,x)-u(\cdot,x)\|_{H^1(\mathbb{T})}\rightarrow0 \quad \text { as }u_n\rightarrow u \text { in }  H^{s}\cap H^{0}.
\end{align*}
Then using  the continuity property in Lemma \ref{lemma2.5} and the compactness of $\mathbb{T}$ yields that
\begin{align*}
\textstyle\|f(t,x,u_{n})-f(t,x,u)\|_{0}\leq C\max_{x\in\mathbb{T}}\|f(\cdot,x,u_{n}(\cdot,x))-f(\cdot,x,u(\cdot,x))\|_{H^1(\mathbb{T})}\rightarrow0
\end{align*}
if $ u_n$ tends to  $u$ in $ H^{s}\cap H^{0}$.

Suppose that \eqref{E3.0} could hold for $\mathrm{p}=k$, with $k\in \mathbb N^+$.  We will show that  it follows for $\mathrm{p}=k+1$, where $k+1 \leq \ell-1$. Then the assumption for $\mathrm{p}=k$ implies that for all $u\in H^s\cap H^{k+1}$,
\begin{align}\label{E2.0}
\|\partial_{x}f(t,x,u)\|_{k}\leq C(k,\|u\|_{s})(1+\|u\|_{k}),\quad\|\partial_{u}f(t,x,u)\|_{k}\leq C(k,\|u\|_{s})(1+\|u\|_{k}).
\end{align}
If we set $\rho(t,x):=f(t,x,u(t,x))$, then $\rho(t,x)=\sum_{j}\rho_{j}(t)e^{{\rm i}j x}$. Observe that  $\partial_{x}\rho(t,x)=\sum_{j}\mathrm{i}j\rho_{j}(t)e^{{\rm i}j x}$. Then the definition of $s$--norm (recall \eqref{E2.2}) yields that
\begin{align*}
\textstyle\|\rho\|^2_{k+1}=&\textstyle\sum_{j}(1+j^{2(k+1)})\|\rho_{j}\|^2_{H^1}=
\sum_{j}\|\rho_{j}\|^2_{H^1}+\sum_{j}j^{2k}j^2\|\rho_j\|^2_{H^1}\\
=&\textstyle\sum_{j}\|\rho_{j}\|^2_{H^1}+\sum_{j}j^{2k}\|(\mathrm{i}j)\rho_j\|^2_{H^1}
\leq(\|\rho\|_{0}+\|\partial_{x}\rho\|_{k})^2.
\end{align*}
This leads to
\begin{align}\label{E2.8}
\|f(t,x,u)\|_{k+1}\leq \|f(t,x,u)\|_{0}+\|\partial_{x}f(t,x,u)\|_{k}+\|\partial_{u}f(t,x,u)\partial_{x}u\|_{k}.
\end{align}

If $\mathrm{p}=1$ ($k=0$), then it follows from \eqref{E2.00}--\eqref{E2.7} and \eqref{E2.8} that
\begin{align*}
\|f(t,x,u)\|_{1}\leq&\textstyle \|f(t,x,u)\|_{0}+\|\partial_{x}f(t,x,u)\|_{0}+
C\max_{x\in\mathbb{T}}\|\partial_{u}f(\cdot,x,u(\cdot,x))\|_{H^1(\mathbb{T})}\|\partial_{x}u\|_{0}\\
\leq&\textstyle 2C(\|u\|_{s})+C'(\|u\|_{s})\|u\|_{1}\leq C(1,\|u\|_{s})(1+\|u\|_{1}).
\end{align*}
In addition, denote $\mathfrak{s}\in(\frac12,\min(1,s))$. It is easy to see that
\begin{align*}
\begin{cases}
\mathfrak{{s}}<k<\mathfrak{s}+1<k+1,\quad k=1,\\
\mathfrak{{s}}<\mathfrak{s}+1<k<k+1,\quad  k\geq2.
\end{cases}
\end{align*}
Combining this with \eqref{E2.3} shows that
\begin{align}\label{E2.9}
\|u\|_{k}\|u\|_{{\mathfrak{s}+1}}\leq\|u\|_{k+1}\|u\|_{\mathfrak{s}}\leq\|u\|_{k+1}\|u\|_{s}.
\end{align}
As a consequence, by \eqref{E2.10}, \eqref{E2.00}--\eqref{E2.9}, we deduce
\begin{align*}
\|f(t,x,u)\|_{k+1}\leq
 & C(\|u\|_{s})+C(k,\|u\|_{s})(1+\|u\|_{k})+C(k)\|\partial_{u}f(t,x,u)\|_{k}
\|u_{x}\|_{L^{\infty}(\mathbb{T};H^1(\mathbb{T}))}\\
&+C(k)
\|\partial_{u}f(t,x,u)\|_{L^{\infty}(\mathbb{T};H^1(\mathbb{T}))}\|u\|_{k+1}\\
\leq &C(\|u\|_{s})+C(k,\|u\|_{s})(1+\|u\|_{k})+C(k)C(k,\|u\|_{s})(1+\|u\|_{k})\|u\|_{{\mathfrak{s}+1}}\\
&+C(k)C(\|u\|_{s})\|u\|_{k+1}\\
\leq&C(k+1,\|u\|_{s})(1+\|u\|_{k+1}).
\end{align*}
Hence \eqref{E3.0} follows for $\mathrm{p}=k+1$.

On the other hand,  suppose that \eqref{E3.1} could hold for $\mathrm{p}=k$. From formula \eqref{E2.8}, we get the continuity property of $f$ with respect to $u$  for  $\mathrm{p}=k+1$ with $k+1\leq \ell-1$.

If $s'$ is not an integer, the conclusion follows from the  Fourier dyadic decomposition. The argument is similar to the proof of Lemma A.1 in \cite{Delort2011}.
\end{proof}
\begin{lemm}\label{lemma2.4}
Let $(x,u)\longmapsto f(\cdot,x,u)$ be in $C^\ell(\mathbb{T}\times\mathbb{R};H^1(\mathbb{T}))$ with $\ell\geq3$. For all $0\leq s'\leq \ell-3$, define the mapping
\begin{align*}
F:H^{s}\cap H^{s'}&\longrightarrow {H}^{s'},\quad u\longmapsto f(t,x,u).
\end{align*}
Then $F$ is a  $C^2$ mapping with respect to $u$ and satisfies that for all $h\in H^s\cap H^{s'}$,
\begin{align*}
{\rm D}F(u)[h]=\partial_{u}f(t,x,u)h, \quad {\rm D}^2F(u)[h,h]=\partial^2_{u}f(t,x,u)h^2, 
\end{align*}
where
\begin{align}\label{E2.5}
\|\partial_{u}f(t,x,u)\|_{{s'}}\leq C(s',\|u\|_{s})(1+\|u\|_{{s'}}),\quad\|\partial^2_{u}f(t,x,u)\|_{{s'}}\leq C(s',\|u\|_{s})(1+\|u\|_{{s'}}).
\end{align}
\end{lemm}
\begin{proof}
For $\ell\geq3$, observe that
\begin{align*}
(x,u)\longmapsto \partial_{u}f(\cdot,x,u)\in C^{\ell-1}(\mathbb{T}\times\mathbb{R};H^1(\mathbb{T})),\quad (x,u)\longmapsto \partial^2_{u}f(\cdot,x,u)\in C^{\ell-2}(\mathbb{T}\times\mathbb{R};H^1(\mathbb{T})).
\end{align*}
In view of Lemma \ref{lemma2.3}, the mappings $u\longmapsto\partial_{u}f(t,x,u)$, $u\longmapsto\partial^2_{u}f(t,x,u)$ are continuous and satisfy two estimates in \eqref{E2.5}.

It remains to verify that $F$  is $C^2$ with respect to $u$. Using the continuity property of $u\longmapsto\partial_{u}f(t,x,u)$ yields that
\begin{align*}
&\|f(t,x,u+h)-f(t,x,u)-\partial_{u}f(t,x,u)h\|_{{s'}}=\textstyle\|h\int_0^1(\partial_{u}f(t,x,u+\mathfrak{v} h)-\partial_{u}f(t,x,u))~\mathrm{d}\mathfrak{v}\|_{{s'}}\\
&\leq \textstyle C(s')\|h\|_{{\max{\{s,s'\}}}}\max_{\mathfrak{v}\in[0,1]}\|\partial_{u}f(t,x,u+\mathfrak{v} h)-\partial_{u}f(t,x,u)\|_{{\max{\{s,s'\}}}}\\
&=\textstyle o(\|h\|_{{\max{\{s,s'\}}}}).
\end{align*}
This implies that ${\rm D}F(u)[h]=\partial_{u}f(t,x,u)h$ and that $u\longmapsto {\rm D}F(u)$ is continuous. In addition,
\begin{align*}
&\textstyle\partial_{u}f(t,x,u+h)h-\partial_{u}f(t,x,u)h-\partial^2_{u}f(t,x,u)h^2=h^2\int_0^1(\partial^2_{u}f(t,x,u+\mathfrak{v} h)-\partial^2_{u}f(t,x,u))~\mathrm{d}\mathfrak{v}.
\end{align*}
By the similar argument as above, $F$ is twice differentiable with respect to $u$ and that $u\longmapsto {\rm D}^2_uF(u)$ is continuous.

Thus this completes the proof.
\end{proof}


\begin{thebibliography}{10}

\bibitem{Barbu1996Periodic}
V.~Barbu and N.~H. Pavel.
\newblock Periodic solutions to one-dimensional wave equation with piece-wise
  constant coefficients.
\newblock {\em J. Differential Equations}, 132(2):319--337, 1996.

\bibitem{Barbu1997Periodic}
V.~Barbu and N.~H. Pavel.
\newblock Periodic solutions to nonlinear one-dimensional wave equation with
  $x$-dependent coefficients.
\newblock {\em Trans. Amer. Math. Soc.}, 349(5):2035--2048, 1997.

\bibitem{berti2006cantor}
M.~Berti and P.~Bolle.
\newblock Cantor families of periodic solutions for completely resonant
  nonlinear wave equations.
\newblock {\em Duke Math. J.}, 134(2):359--419, 2006.

\bibitem{berti2008cantor}
M.~Berti and P.~Bolle.
\newblock Cantor families of periodic solutions of wave equations with ${C}^k$
  nonlinearities.
\newblock {\em No{DEA} {N}onlinear differ. equ. appl.}, 15(1-2):247--276, 2008.

\bibitem{berti2015abstract}
M.~Berti, L.~Corsi, and M.~Procesi.
\newblock An abstract {N}ash-{M}oser theorem and quasi-periodic solutions for
  {NLW} and {NLS} on compact {L}ie groups and homogeneous manifolds.
\newblock {\em Comm. Math. Phys.}, 334(3):1413--1454, 2015.

\bibitem{bourgain1994construction}
J.~Bourgain.
\newblock Construction of quasi-periodic solutions for {H}amiltonian
  perturbations of linear equations and applications to nonlinear {PDE}.
\newblock {\em Internat. Math. Res. Notices}, 1994(11):475--497, 1994.

\bibitem{bourgain1995construction}
J.~Bourgain.
\newblock Construction of periodic solutions of nonlinear wave equations in
  higher dimension.
\newblock {\em Geom. Funct. Anal.}, 5(4):629--639, 1995.

\bibitem{calleja2017response}
R.~C. Calleja, A.~Celletti, L.~Corsi, and R.~de~la Llave.
\newblock Response solutions for quasi-periodically forced, dissipative wave
  equations.
\newblock {\em SIAM J. Math. Anal.}, 49(4):3161--3207, 2017.

\bibitem{calleja2013construction}
R.~C. Calleja, A.~Celletti, and R.~de~la Llave.
\newblock Construction of response functions in forced strongly dissipative
  systems.
\newblock {\em Discrete Contin. Dyn. Syst.}, 33(10):4411--4433, 2013.

\bibitem{Colombini1979Sur}
F.~Colombini, E.~De~Giorgi, and S.~Spagnolo.
\newblock Sur les \'{e}quations hyperboliques avec des coefficients qui ne
  d\'{e}pendent que du temps.
\newblock {\em Ann. {S}cuola {N}orm. {S}up. {P}isa {C}l. {S}ci. (4)},
  6(3):511--559, 1979.

\bibitem{Colombini2015hyperbolic}
F.~Colombini and G.~M\'{e}tivier.
\newblock Counterexamples to the well posedness of the {C}auchy problem for
  hyperbolic systems.
\newblock {\em Anal. {PDE}}, 8(2):499--511, 2015.

\bibitem{craig1993newton}
W.~Craig and C.~E. Wayne.
\newblock Newton's method and periodic solutions of nonlinear wave equations.
\newblock {\em Comm. Pure Appl. Math.}, 46(11):1409--1498, 1993.

\bibitem{Reissig2011hyperbolic}
M.~D'Abbicco and M.~Reissig.
\newblock Long time asymptotics for 2 by 2 hyperbolic systems.
\newblock {\em J. Differential Equations}, 250(2):752--781, 2011.

\bibitem{Delort2011}
J.-M. Delort.
\newblock Periodic solutions of nonlinear {S}chr\"odinger equations: a
  paradifferential approach.
\newblock {\em Anal. PDE}, 4(5):639--676, 2011.

\bibitem{Eliasson2016beam}
L.~H. Eliasson, B.~Gr\'{e}bert, and S.~B. Kuksin.
\newblock {KAM} for the nonlinear beam equation.
\newblock {\em Geom. Funct. Anal.}, 26(6):1588--1715, 2016.

\bibitem{Gao2009}
Y.~Gao, Y.~Li, and J.~Zhang.
\newblock Invariant tori of nonlinear {S}chr\"odinger equation.
\newblock {\em J. Differential Equations}, 246(8):3296--3331, 2009.

\bibitem{gauckler2016long}
L.~Gauckler, E.~Hairer, and C.~Lubich.
\newblock Long-term analysis of semilinear wave equations with slowly varying
  wave speed.
\newblock {\em Comm. Partial Differential Equations}, 41(12):1934--1959, 2016.

\bibitem{gentile2019forced}
G.~Gentile, A.~Mazzoccoli, and F.~Vaia.
\newblock Forced quasi-periodic oscillations in strongly dissipative systems of
  any finite dimension.
\newblock {\em Commun. Contemp. Math.}, 21(7):1850064, 22, 2019.

\bibitem{Procesi2017KAM}
E.~{H}aus and M.~{P}rocesi.
\newblock {KAM} for beating solutions of the quintic {NLS}.
\newblock {\em Comm. {M}ath. {P}hys.}, 354(3):1101--1132, 2017.

\bibitem{Hirosawa2007Wave}
F.~Hirosawa.
\newblock On the asymptotic behavior of the energy for the wave equations with
  time depending coefficients.
\newblock {\em Math. {A}nn.}, 339(4):819--838, 2007.

\bibitem{Hirosawa2010wave}
F.~Hirosawa.
\newblock Energy estimates for wave equations with time dependent propagation
  speeds in the gevrey class.
\newblock {\em J. Differential Equations}, 248(12):2972--2993, 2010.

\bibitem{ji2006periodic}
S.~Ji and Y.~Li.
\newblock Periodic solutions to one-dimensional wave equation with
  {$x$}-dependent coefficients.
\newblock {\em J. Differential Equations}, 229(2):466--493, 2006.

\bibitem{ji2011time}
S.~Ji and Y.~Li.
\newblock Time periodic solutions to the one-dimensional nonlinear wave
  equation.
\newblock {\em Arch. Ration. Mech. Anal.}, 199(2):435--451, 2011.

\bibitem{zettl1996eigenvalues}
Q.~Kong and A.~Zettl.
\newblock Eigenvalues of regular {S}turm-{L}iouville problems.
\newblock {\em J. Differential Equations}, 131(1):1--19, 1996.

\bibitem{Kuchment1982}
P.~A. Kuchment.
\newblock Floquet theory for partial differential equations.
\newblock {\em Uspekhi Mat. Nauk}, 37(4(226)):3--52, 240, 1982.

\bibitem{kuksin1987hamiltonian}
S.~B. Kuksin.
\newblock Hamiltonian perturbations of infinite-dimensional linear systems with
  imaginary spectrum.
\newblock {\em Funktsional. Anal. i Prilozhen.}, 21(3):22--37, 95, 1987.

\bibitem{Levitan1991}
B.~M. Levitan and I.~S. Sargsjan.
\newblock {\em Sturm-{L}iouville and {D}irac operators}, volume~59 of {\em
  Mathematics and its Applications (Soviet Series)}.
\newblock Kluwer Academic Publishers Group, Dordrecht, 1991.
\newblock Translated from the Russian.

\bibitem{marchenko1977sturm}
V.~Marchenko.
\newblock Sturm-liouville operators and their applications.
\newblock {\em Kiev Izdatel Naukova Dumka}, 1977.

\bibitem{Procesi2015}
C.~Procesi and M.~Procesi.
\newblock A {KAM} algorithm for the resonant non-linear {S}chr\"odinger
  equation.
\newblock {\em Adv. Math.}, 272:399--470, 2015.

\bibitem{Rabinowitz1967periodic}
P.~H. Rabinowitz.
\newblock Periodic solutions of nonlinear hyperbolic partial differential
  equations.
\newblock {\em Comm. Pure Appl. Math.}, 20:145--205, 1967.

\bibitem{rabinowitz1968periodic}
P.~H. Rabinowitz.
\newblock Periodic solutions of nonlinear hyperbolic partial differential
  equations. {II}.
\newblock {\em Comm. Pure Appl. Math.}, 22:15--39, 1968.

\bibitem{kh1994}
S.~Spagnolo.
\newblock The {C}auchy problem for {K}irchhoff equations.
\newblock In {\em Proceedings of the {S}econd {I}nternational {C}onference on
  {P}artial {D}ifferential {E}quations ({I}talian) ({M}ilan, 1992)}, volume~62,
  pages 17--51 (1994), 1992.

\bibitem{titchmarsh1958eigenfunction}
E.~C. Titchmarsh.
\newblock {\em Eigenfunction expansions associated with second-order
  differential equations. {V}ol. 2}.
\newblock Oxford, at the Clarendon Press, 1958.

\bibitem{wayne1990periodic}
C.~E. Wayne.
\newblock Periodic and quasi-periodic solutions of nonlinear wave equations via
  {KAM} theory.
\newblock {\em Comm. Math. Phys.}, 127(3):479--528, 1990.

\bibitem{yagdjian2005global}
K.~Yagdjian.
\newblock Global existence in the {C}auchy problem for nonlinear wave equations
  with variable speed of propagation.
\newblock In {\em New trends in the theory of hyperbolic equations}, volume 159
  of {\em Oper. Theory Adv. Appl.}, pages 301--385. Birkh\"{a}user, Basel,
  2005.

\end{thebibliography}


\end{document}